\RequirePackage{fix-cm}
\documentclass{svjour3}                     
\smartqed  
\usepackage{graphicx}
\usepackage{biograph}        
\usepackage{amssymb,amsmath,booktabs}
\usepackage{graphicx}
\usepackage{times}
\usepackage{hyperref}

\newcommand{\dx}{\;{\rm d}\xi}
\newcommand{\ds}{\;{\rm d}s}
\def\RR{\mathbb{R}}
\def\SS{\mathbb{S}}

\begin{document}
	
\title{Decomposition of arrow type positive semidefinite matrices with application to topology optimization}
\titlerunning{Decomposition of arrow type positive semidefinite matrices}

\author{Michal Ko\v{c}vara}

\institute{Michal Ko\v{c}vara\at School of Mathematics, University of
Birmingham, Birmingham B15 2TT, UK and Institute of Information Theory
and Automation, Academy of Sciences of the Czech Republic, Pod
vod\'arenskou v\v{e}\v{z}\'{\i}~4, 18208 Praha 8,
Czech Republic, \email{m.kocvara@bham.ac.uk}
}

\date{Received: date / Accepted: date}

\maketitle

\begin{abstract}
Decomposition of large matrix inequalities for matrices with chordal sparsity graph has been recently used by Kojima et al.\ \cite{kim2011exploiting} to reduce problem size of large scale semidefinite optimization (SDO) problems and thus increase efficiency of standard SDO software. A by-product of such a decomposition is the introduction of new dense small-size matrix variables. We will show that for arrow type matrices satisfying suitable assumptions, the additional matrix variables have rank one and can thus be replaced by vector variables of the same dimensions. This leads to significant improvement in efficiency of standard SDO software. We will apply this idea to the problem of topology optimization formulated as a large scale linear semidefinite optimization problem. Numerical examples will demonstrate tremendous speed-up in the solution of the decomposed problems, as compared to the original large scale problem. In our numerical example the decomposed problems exhibit linear growth in complexity, compared to the more than cubic growth in the original problem formulation. We will also give a connection of our approach to the standard theory of domain decomposition and show that the additional vector variables are outcomes of the corresponding discrete Steklov-Poincar\'{e} operators.
\end{abstract}

\section{Introduction}
General purpose algorithms and software for semidefinite optimization (SDO) are dominated by interior point and barrier type methods. Any such software exhibits two bottlenecks regarding computational complexity, and thus CPU time, and memory requirements. The first one is the evaluation of the system matrix (Schur complement matrix or Hessian of augmented Lagrangian) in every step of the underlying Newton method. The second one is then the solution of a linear system with this matrix. For problems with large matrix inequalities, it is often the first bottleneck that dominates the CPU time and that prevents the user from solving large scale problems.

To circumvent this obstacle, the technique of decomposition of a large matrix inequality into several smaller ones proved to be efficient, at least for certain classes of problems.
Decomposition of positive semidefinite matrices with a certain sparsity pattern was first investigated in Agler et al.\ \cite{agler} and, independently, by Griewank and Toint \cite{griewank-toint}. An extensive study has been recently published by Vandenberghe and Andersen \cite{vandenberghe2015chordal}. We will call this technique \emph{chordal decomposition}. It was first used in semidefinite optimization by Kojima and his co-workers; see \cite{fukuda2001exploiting,nakata2003exploiting} and, more recently, \cite{kim2011exploiting}. The group also developed a preprocessing software for semidefinite optimization named SparseCoLO \cite{sparsecolo} that performs the decomposition of matrix constraints  automatically.

The goal of this paper is twofold. Firstly, we introduce a new decomposition of arrow type positive semidefinite matrices called \emph{arrow decomposition}. Unlike the chordal decomposition that generates additional dense matrix variables, arrow decomposition only requires additional vector variables of the same size, leading to significant reduction of number of variables in the decomposed problem. The second goal is to apply both decomposition techniques to the topology optimization problem. This problem arises from finite element discretization of a partial differential equation. We will show that techniques known from domain decomposition can be used to define the matrix decomposition. In particular, we will be able to control the number and size of the decomposed matrix inequalities. 
We will also give a connection of the arrow decomposition with the theory of domain decomposition and show that the additional vector variables are outcomes of the corresponding discrete Steklov-Poincar\'{e} operators.

To solve all semidefinite optimization problems, we will use the state of the art solver MOSEK \cite{mosek}. Numerical examples will demonstrate tremendous speed-up in the solution of the decomposed problems, as compared to the original large scale problem. Moreover, in our numerical examples the arrow decomposition exhibits linear growth in complexity, compared to the higher than cubic growth when solving the original problem formulation.

\paragraph{Notation}
Let $\SS^n$ be the space of $n\times n$ symmetric matrices, $A\in\SS^n$,
and $I\subset\{1,\ldots,n\}$ with $s=|I|$. We denote 
\begin{itemize}
	\item by $(A)_{i,j}$ the $(i,j)$-th element of $A$;
	\item by $(A)_I$ the restriction of $A$ to $\SS^s$, i.e., the $s\times s$ submatrix of $A$ with row and column indices from $I$ ;
	\item by $O_{m,n}$ the $m\times n$ zero matrix; when the dimensions are clear from the context, we simply use $O$.
\end{itemize} 
A matrix is called \emph{dense} if all its elements are non-zeros. Otherwise, the matrix is called \emph{sparse}. A matrix-valued function $A(x)$ is called dense if there exists $\bar{x}$ such that $A(\bar{x})$ is dense.

Let $A\in\SS^n$. The undirected graph $G(N,E)$ with $N=\{1,\ldots,n\}$ is called \emph{sparsity graph of $A$} (or just \emph{graph of $A$}) when $(i,j)\in E$ if and only if $(A)_{i,j}\not= 0$.

For an index set $I\subset \{1,\ldots,n\}$ we define
\begin{align*}
\SS^n(I)&:=\{Y\in\SS^n \mid  (Y)_{i,j}=0 \mbox{~if~} (i,j)\not\in I\times I\}\\
\SS^n_+(I)&:=\{Y\in\SS^n(I) \mid Y\succeq 0\}\,.
\end{align*}
Furthermore, let $G(N,E)$ be an undirected graph with $N=\{1,\ldots,n\}$ and edge set $E\subseteq N\times N$. We define
	$$\SS^n(G) := \{Y\in\SS^n \mid (Y)_{ij}=0 \mbox{~if~} (i,j)\not\in E\cup\{(i,i)\}\}$$
and analogously $\SS^n_+(G)$.

Let $G_s(N_s,E_s)$ be an induced subgraph of $G(N,E)$.
Notice the difference between $\SS^n(G_s)$ and $\SS^n(N_s)$. If $A\in \SS^n(N_s)$ then its restriction $(A)_{N_s}$ is a dense matrix. This is not true for $A\in \SS^n(G_s)$, the sparsity pattern of which is given by the set of edges $E_s$. In particular, $\SS^n(G_s) = \SS^n(N_s)$ if and only if $G_s$ is a maximal clique.

Finally, for functions from $\RR^d\to\RR$ we will use bold italics (such as $\boldsymbol{u}$ or $\boldsymbol{u}(\xi)$), while for vectors resulting from finite element discretization of these functions, we will use the same symbol but in italics (e.g. $u\in\RR^n$).

\section{Decomposition of positive semidefinite matrices}\label{sec:deco}
\subsection{Matrices with chordal graphs}
We first recall the well-studied case of matrices with chordal sparsity graph.
The following theorem was proved independently by Grone, et al.\ \cite{grone},
Griewank and Toint \cite{griewank-toint} and by Agler et al.\ \cite{agler}. A new, shorter proof can be found in~\cite{kakimura}. 
\begin{theorem}\label{th:agler}
	Let $G(N,E)$ be an undirected graph with maximal cliques $C_1,\ldots,C_p$.
	The following two statements are equivalent:
	\begin{itemize}
		\item[(i)] $G(N,E)$ is chordal.
		\item[(ii)] For any $A\in\SS^n(G)$, $A\succeq 0$, there are matrices
		$Y_k\in\SS_+^n(C_k)$, $k=1,\ldots,p$, such that $A =
		Y_1+Y_2+\ldots+Y_p$\,.
	\end{itemize}
\end{theorem}
Notice that this decomposition is not unique. However, Kakimura \cite{kakimura} has shown that there exist matrices $Y^*_k$ minimizing $\sum\limits_{k=1}^p{\rm rank}\,Y_k$ subject to $\sum\limits_{k=1}^p Y_k = A$ and $Y_k\in\SS_+^n(C_k)$ ($k=1,\ldots,p$) and that $\sum\limits_{k=1}^p {\rm rank}\,Y^*_k = {\rm rank}\,A $.

\subsection{Matrices embedded in those with a chordal graph}
Let $A\in\SS^n$, $n\geq 3$,  with a sparsity graph $G=(N,E)$. Let the set of nodes $N=\{1,2,\ldots,n\}$ be partitioned into $p\geq 2$ overlapping sets
$$
N=I_1\cup I_2\cup\ldots\cup I_p \,.
$$
Let $I_{k,\ell}$ denote the intersection of the $k$th and $\ell$th set, i.e.,
$$
I_{k,\ell} := I_k\cap I_{\ell}\,,\quad (k,\ell)\in\Theta_p
$$
with
$$
\Theta_p:=\{(i,j)\mid i=1,\ldots,p-1;\ j=2,\ldots,p;\ i<j\}\,.
$$

\paragraph{Assumption 1.} Let $1\leq k\leq p$. There exist at least one index $\ell$ with $1\leq \ell\leq p$,\ $\ell\not= k$, such that $I_k\cap I_{\ell}\not=\emptyset$.

\paragraph{Assumption 2.} $I_k\cup I_\ell \not= I_k$ for all  $1\leq k,\ell\leq p$, $k\not=\ell$, i.e., no $I_\ell$ is a subset of any $I_k$.

\paragraph{Assumption 3.} The intersections are ``sparse" in the sense that for each $k\in\{1,\ldots,p\}$ there are at most $p_k$ indices $\ell_i$ such that 
$I_k\cap I_{\ell_i}\not=\emptyset$,\ $i=1,\ldots,p_k$, where $1\leq p_k\ll p$.

\medskip
In a typical situation only $I_{k,k+1}$, $k=1,\ldots,p-1$, are not empty (corresponding to a block diagonal matrix with overlapping blocks) or $I_k$ has a non-empty intersection with up to eight other sets (see Section~\ref{sec:topodec}).

Denote the induced subgraphs of $G(N,E)$ corresponding to $I_k$ by $G_k(I_k,E_k)$, $k=1,\ldots,p$. \emph{These subgraphs are not necessarily cliques.}
Assume that 
$$
A = \sum_{k=1}^p Q_k \quad\mbox{where}\quad Q_k\in\SS^n(G_k)\,.
$$
For all $k=1,\ldots,p$, let $\widehat{G}_k(I_k,\widehat{E}_k)$ denote a completion of $G_k(I_k,E_k)$, i.e., a clique in $G(N,E)$. According to Assumption 2, $\widehat{G}_k(I_k,\widehat{E}_k)$ are even maximal cliques. Clearly, $Q_k\in\SS^n(\widehat{G}_k)$.

\paragraph{Assumption 4.} The union
$\widehat{G}(N,\widehat{E}) := \bigcup\limits_{k=1}^p \widehat{G}_k(I_k,\widehat{E}_k)$ is a chordal graph.
 
\medskip
The graph $\widehat{G}(N,\widehat{E}) \supset G(N,E)$ is called a \emph{chordal extension} of $G(N,E)$; see, e.g., \cite[Section 8.3]{vandenberghe2015chordal}.

Notice that the rather restrictive Assumption~4 is satisfied when $A$ is a block diagonal matrix with overlapping blocks. It may \emph{not} be satisfied in the application in Section~\ref{sec:topodec}; we will see, however, that it will not be needed in this application.

\begin{theorem}\label{th:agler1}
	Let Assumptions 1--4 hold. 
	The following two statements are equivalent:
	\begin{itemize}
		\item[(i)] $A\succeq 0$.
		\item[(ii)] There exist matrices $S_{k,\ell}\in\SS^n(I_{k,\ell}),\ (k,\ell)\in\Theta_p$, such that $$A = \displaystyle\sum_{k=1}^p \widetilde{Q}_k
		\mbox{~with~} \widetilde{Q}_k = Q_k -\sum_{\ell:\ell<k}S_{\ell,k}+\sum_{\ell:\ell>k}S_{k,\ell}$$ 
		and $$\widetilde{Q}_k\succeq 0\quad k=1,\ldots,p.$$ 
		If $I_{k,\ell}=\emptyset$ or is not defined then $S_{k,\ell}$ is a zero matrix.
	\end{itemize}
\end{theorem}
\begin{proof}
	Using the chordal extension $\widehat{G}(N,\widehat{E})$ of $G(N,E)$, we embed
	the matrix $A$ into a set of matrices with chordal sparsity graphs with maximal cliques
	$\widehat{G}_k(I_k,\widehat{E}_k)$, $k=1,\ldots,p$. Then we can apply Theorem \ref{th:agler}. Hence there exist matrices $Y_k\in\SS_+^n(I_k)$, $k=1,\ldots,p$, such that  $A=Y_1+\ldots+Y_p$. Now, $Y_k$ must be equal to $Q_k$ for the ``internal'' indices of $I_k$, i.e., for all $(i,j)\in \left(I_k\setminus \left(\bigcup_{\ell:\ell>k}(I_{k,\ell})\cup \bigcup_{\ell:\ell<k}(I_{\ell,k})\right)\right)^2$. Therefore the unknown elements of $Y_k$ reduce to the overlaps $I_{k,\ell}$. 
	
	Having $Q_k$ and $Y_k$, $k=1,\ldots,p$, we will now define matrices the $S_{k,\ell}$ as follows. Firstly, for $k=1$ we select any solution $\left\{S_{1,\ell}\right\}_{I_{1,\ell}\not=\emptyset}$ of the equation
	$$
	  Y_1 = Q_1 + \sum_{\substack{\ell:\ell>1\\I_{1,\ell}\not=\emptyset}}S_{1,\ell}\,.
	$$
	Notice that many elements of matrices $S_{1,\ell}$ $(I_{1,\ell\not=\emptyset})$ are uniquely defined by this equation. Only elements with indices from nonempty intersections $I_{1,\ell}\cap I_{1,k}$ are not unique, as they appear in more than one matrix $S_{\bullet,\bullet}$ in the above equation.
	
	Now, for $1<k< p$, we solve the equation 
	$$
	  Y_k = Q_k - \sum_{\substack{\ell:\ell<k\\I_{\ell,k}\not=\emptyset}}S_{\ell,k} + \sum_{\substack{\ell:\ell>k\\I_{k,\ell}\not=\emptyset}}S_{k,\ell}\,.
	$$
	All matrices $S_{\ell,k},\ \ell<k,$ were defined in steps $1,\ldots,k-1$, hence we are in the same situation as above and select any solution $\left\{S_{k,\ell}\right\}_{\ell>k,I_{k,\ell}\not=\emptyset}$ of the above equation. Any selection of the non-unique elements of $S_{\bullet,\bullet}$ will be consistent with the last equation 
	$$
	  Y_p = Q_p - \sum_{\substack{\ell:\ell<p\\I_{\ell,p}\not=\emptyset}}S_{\ell,p}
	$$
	because we know that $A = \sum_{k=1}^p Y_k = \sum_{k=1}^p Q_k$.
	Therefore $A = \sum_{k=1}^p\widetilde{Q}_k$ and the assertion follows.\qed
\end{proof}

\subsection{Arrow type matrices}
Let us now consider a particular type of sparse matrices, the arrow type matrices.
Let again $A\in\SS^n$, $n\geq 3$, and let $I_k$, $I_{k,\ell}$ and $G_k(I_k,E_k)$, $k=1,\ldots,p$, be defined as in the previous section.

Assume again that $A$ is a sum of matrices associated with $G_k$:
$$
  A = \sum_{k=1}^p A_k,\quad A_k\in\SS^n(G_k)\,.
$$
Further, let $B\in\RR^{n\times m}$, $B=\sum\limits_{k=1}^p B_k$ with $B_k$, $k=1,\ldots,p$, being rectangular matrices such that
$$
  (B_k)_{i,j}=0\quad \mbox{for}\ i\not\in I_k 
$$
and assume that 
$$
  m < \min_{\substack{k,\ell=1,\ldots,p\\k<\ell}} |I_{k,\ell}|\,.
$$
We also define
$$
  {\widehat I}_{k} = I_{k} \cup \{n+1,\ldots,n+m\}\,,\quad k=1,\ldots,p
$$
and
$$
{\widehat I}_{k,\ell} = I_{k,\ell} \cup \{n+1,\ldots,n+m\}\,,\quad (k,l)\in\Theta_p\,.
$$

Finally, let $C\in\SS^m$ be positive definite. We define the following \emph{arrow type matrix}: 
\begin{equation}\label{eq:M}
  M =\sum_{k=1}^p M_k + \begin{bmatrix} 0&0\\0&C \end{bmatrix}\quad \mbox{where}\  M_k=\begin{bmatrix} A_k & B_k\\ B_k^\top& 0 \end{bmatrix},\ k=1\ldots,p    \,.
\end{equation}
Let us recall that
\begin{equation}\label{eq:sp}
  (M_k)_{i,j} = 0\ \mbox{for}\ (i,j)\notin {\widehat I}_k\times {\widehat I}_k,\quad k=1,\ldots,p\,.
\end{equation}

The simplest example of an arrow type matrix is a block diagonal matrix with overlapping blocks and with additional rows and columns corresponding to matrices $B$ and $C$; see Figure~\ref{fig:arrow type}. Notice that the matrices $A_k$ can also be sparse. 

Notice, however, that the structure of the overlapping blocks can be more complicated and that, in general, $A$ (the arrow ``shaft") does not have to be a band matrix. Such matrices arise in the application introduced later in Section~\ref{sec:topo}; see Figure~\ref{fig:mat_glob} and \ref{fig:mat14}.
\begin{figure}[h]
	\begin{center}
		\resizebox{0.3\hsize}{!}
		{\includegraphics{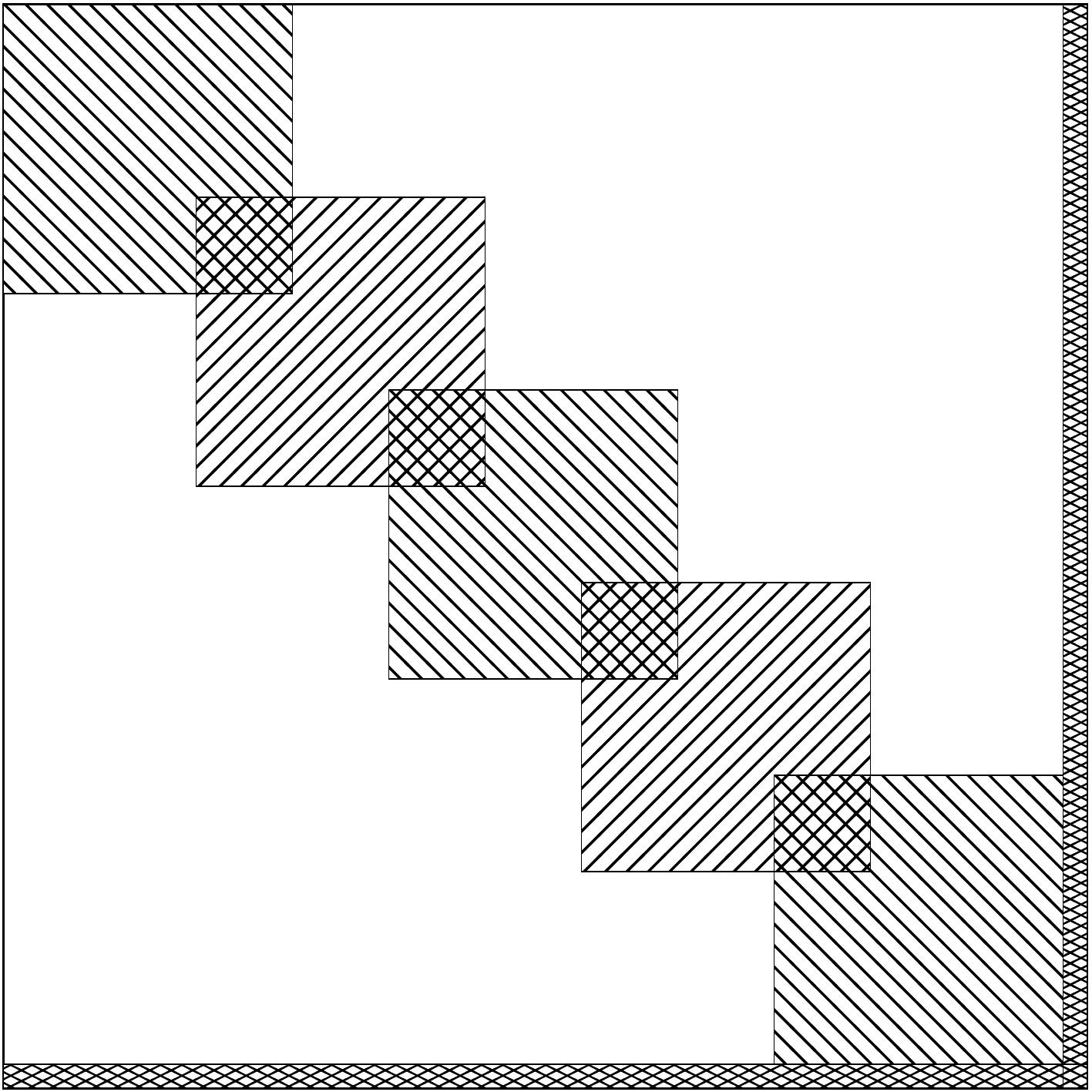}}
	\end{center}
	\caption{An example of an arrow type matrix with overlapping blocks%
		.}\label{fig:arrow type}
\end{figure}
In this application, we will have $m=1$, so that $B$ will be an $n$-vector and $C\in\RR$. However, in this section we consider the more general situation which may be useful in other applications.
We will first adapt Theorem~\ref{th:agler1} to the arrow type structure.
\begin{corollary}\label{th:coro}
	Let Assumptions 1--4 hold.
	Let $M$ be defined as in (\ref{eq:M}).
	The following two statements are equivalent:
	\begin{itemize}
		\item[(i)] $M\succeq 0$\,.
		\item[(ii)] There exist matrices 
		$S_{k,\ell}\in\SS^n(\widehat{I}_{k,\ell}),\ (k,\ell)\in\Theta_p$, such that
		$$M = \displaystyle\sum_{k=1}^p \widetilde{M}_k
		\mbox{~~with~~} \widetilde{M}_k = M_k - \sum_{\ell:\ell<k}S_{\ell,k}+ \sum_{\ell:\ell>k}S_{k,\ell}$$ 
		and $$\widetilde{M}_k\succeq 0\quad k=1,\ldots,p\,.$$ 
		If $I_{k,\ell}=\emptyset$ or is not defined then $S_{k,\ell}$ is a zero matrix.
	\end{itemize}
\end{corollary}
\begin{proof}
A direct application of Theorem~\ref{th:agler1} with $Q_k=M_k$ for $k=1,\ldots,p-1$, and $Q_p =\begin{bmatrix} A_p & B_p\\ B_p^\top& C \end{bmatrix}$.\qed
\end{proof}

Under additional assumptions, we can strengthen the above corollary as follows.
\begin{theorem}\label{th:theo}
	Let Assumptions 1--3 hold.
	Assume that $A_k\succeq 0$, $k=1,\ldots,p$, $A\succ 0$ and $C\succ 0$.
	Let $M$ be defined as in (\ref{eq:M}).  The following two statements are equivalent:
	\begin{enumerate}
		\item[(i)] $M\succeq 0$\,.
		\item[(ii)] There exist matrices $D_{k,\ell}\in\RR^{n\times m}$ such that $(D_{k,\ell})_{i,j}=0$ for $(i,j)\notin I_{k,\ell}\times \{1,\ldots m\}$, $(k,\ell)\in\Theta_p$,
		and matrices $C_k\in\SS^m$, $k=1,\ldots,p$, such that
		$$
		M = \sum_{k=1}^p \widetilde{M}_k ,\ \mbox{with}\ \widetilde{M}_k = M_k - 
		\sum_{\ell:\ell<k}\begin{bmatrix} 0&D_{\ell,k}\\D^\top_{\ell,k}&0 \end{bmatrix}+
		\sum_{\ell:\ell>k}\begin{bmatrix} 0&D_{k,\ell}\\D^\top_{k,\ell}&0 \end{bmatrix}+ \begin{bmatrix} 0&0\\0&C_k \end{bmatrix}
		$$ 
		and
		$$
		\widetilde{M}\succeq 0,\quad k=1,\ldots,p\,.
		$$
		If $I_{k,\ell}=\emptyset$ or is not defined then $D_{k,\ell}$ is a zero  matrix.
	\end{enumerate}
\end{theorem}

\begin{proof}
	We will prove the theorem by constructing matrices $D_{k,k+1}$ and $C_k$.  By assumption, $A$ is positive definite, so that we can define
		\begin{equation}\label{eq:2a}
		  X=A^{-1}B\,,\quad \mbox{i.e.,}\quad \sum_{k=1}^p A_k X = \sum_{k=1}^p B_k\,.
		\end{equation}
		Then 
		\begin{equation}\label{eq:3}
		(A_k X)_{i,j} = (B_k)_{i,j}\quad \mbox{for}\ i\in I_k\setminus \left(\bigcup_{\ell:\ell>k}(I_{k,\ell})\cup \bigcup_{\ell:\ell<k}(I_{\ell,k})\right),\ j=1,\ldots,p\,.
		\end{equation}		  
	We define $D_{k,k+1}$ and $C_k$ as follows. For $k=1$, we solve the equation
	$$
	  A_1X-B_1 = \sum_{\substack{\ell:\ell>1\\I_{1,\ell}\not=\emptyset}}D_{1,\ell}\,.
	$$
	As in the proof of Theorem~\ref{th:agler1}, some elements of thus defined $D_{1,\ell}$ may not be unique; in this case, we just select a solution.
	Then, for any $1<k<p$, we solve the equation 
	$$
	A_kX-B_k = - \sum_{\substack{\ell:\ell<k\\I_{\ell,k}\not=\emptyset}}D_{\ell,k} + \sum_{\substack{\ell:\ell>k\\I_{k,\ell}\not=\emptyset}}D_{k,\ell}
	$$
	to define $D_{k,\ell},\ \ell>k,$ analogously to Theorem~\ref{th:agler1}. Any selection of the non-unique elements of $D_{\bullet,\bullet}$ will be consistent with the last equation 
	$$
	A_pX-B_p = - \sum_{\substack{\ell:\ell<p\\I_{\ell,p}\not=\emptyset}}D_{\ell,p}
	$$
	because of (\ref{eq:2a}). From (\ref{eq:sp}) and (\ref{eq:3}) we see that $D_{{k,\ell}}$ is only non-zero on $I_{k,\ell}$, $(k,\ell)\in\Theta_p$, as required.
	
	Define further 
	\begin{align*}
	&\widehat{C}_k=X^\top A_kX,\ \ k=1,\ldots,p,\\
	&C_k=\widehat{C}_k,\ \  k=1,\ldots,p-1 \quad \mbox{and} \quad C_p = C-\sum_{k=1}^{p-1}C_k\,.
	\end{align*}

	Now the matrices defined for $k=1,\ldots,p$ by
	$$
	  \widehat{M}_k = M_k - \sum_{\substack{\ell:\ell<k\\I_{\ell,k}\not=\emptyset}}\begin{bmatrix} 0&D_{\ell,k}\\D^\top_{\ell,k}&0 \end{bmatrix}+
	  \sum_{\substack{\ell:\ell>k\\I_{k,\ell}\not=\emptyset}}\begin{bmatrix} 0&D_{k,\ell}\\D^\top_{k,\ell}&0 \end{bmatrix}+ \begin{bmatrix} 0&0\\0&\widehat{C}_k \end{bmatrix} = \begin{bmatrix} A_k&A_kX\\X^\top A_k& X^\top A_kX\end{bmatrix}
	$$
	are clearly positive semidefinite with (at least) $m$ zero eigenvalues. We set $\widetilde{M}_k=\widehat{M}_k$, $k=1,\ldots,p-1,$ and $\widetilde{M}_p = M_p - \sum_{\substack{\ell:\ell<p\\I_{\ell,p}\not=\emptyset}}\begin{bmatrix} 0&D_{\ell,p}\\D^\top_{\ell,p}&0 \end{bmatrix} + \begin{bmatrix} 0&0\\0&{C}_p\end{bmatrix}$.	By construction, $M = \sum\limits_{k=1}^p \widetilde{M}_k$.  
	
	It remains to show that $\widetilde{M}_p=\begin{bmatrix} A_p&A_pX\\X^\top\!A_p\  &C-\sum\limits_{k=1}^{p-1}C_k\end{bmatrix}\succeq 0$ whenever $M\succeq 0$.
	As $A_p\succeq 0$ by assumption, positive semidefiniteness of $\widetilde{M}_p$ amounts to
	$$
	   C-\sum\limits_{k=1}^{p-1}C_k - X^\top\!A_pA_p^{-1}A_pX= C-\sum\limits_{k=1}^{p}X^\top\!A_kX\succeq 0
	$$
	which, by (\ref{eq:2a}) is the same as 
	$$ 
	   C-B^\top X\succeq 0\,.
	$$
	By the Schur complement theorem, the last inequality is equivalent to $M\succeq 0$. This completes the proof.\qed
\end{proof}

Let $r$ be the number of non-empty sets $I_{k,\ell}$, $(k,\ell)\in\Theta_p$.
Comparing Corollary~\ref{th:coro} with Theorem~\ref{th:theo} we see that both provide us with a decomposition of a ``large" matrix inequality $M\succeq 0$ by a number of smaller ones $\widetilde{M}_k\succeq 0$, $k=1,\ldots,p$. However, while in Corollary~\ref{th:coro} we have to introduce $r$ additional matrix variables of sizes $|\widehat{I}_{k,\ell}|\times |\widehat{I}_{k,\ell}|$, in Theorem~\ref{th:theo} we only have $r$ additional matrix variables of sizes $|I_{k,\ell}|\times m$ and $p$ matrix variables of size $m\times m$. Recall that $m < \min\limits_{\substack{k,\ell=1,\ldots,p\\k<\ell}}{|I_{k,\ell}|}$ and, in our application below, $m=1$, so the additional variables in Theorem~\ref{th:theo} are vectors instead of matrices of the same dimension in Corollary~\ref{th:coro}, offering thus significant reduction in the dimension of the additional variables.

Notice that in Theorem~\ref{th:theo} we only require Assumptions 1--3 to hold, we do not need the restrictive Assumption~4. This, in turn, means that if $M$ satisfies assumptions of Theorem~\ref{th:theo}, we can apply Corollary~\ref{th:coro} without verifying Assumption~4, because we can choose, by Theorem~\ref{th:theo}, $S_{\ell,k} = \begin{bmatrix} 0&D_{\ell,k}\\D^\top_{\ell,k}&0 \end{bmatrix}$. This, of course, is only true for our specific definition of arrow type matrices.

We will call the decomposition of arrow type matrices using Corollary~\ref{th:coro} \emph{chordal decomposition} and the one using Theorem~\ref{th:theo} \emph{arrow decomposition}.

Two natural questions arise:
\begin{enumerate}
	\item Are the additional assumptions of Theorem~\ref{th:theo} too restrictive? Are there any applications satisfying them?
	\item Is it worth reducing the dimension of the additional variables? Will it bring any significant savings of CPU time when solving the decomposed problem?
\end{enumerate}
Both questions will be answered in the rest of the paper using a problem from structural optimization.

\section{Application: Topology optimization problem, semidefinite formulation}\label{sec:topo}
Consider an elastic body occupying a $d$-dimensional bounded
domain ${\rm \Omega}\subset \RR^d$
with a Lipschitz boundary $\partial{\rm \Omega}$, where $d\in\{2,3\}$.
By $\boldsymbol{u}(\xi) \in \RR^d$ we denote the displacement vector at a point $\xi$, and by
$$
\boldsymbol{e}_{ij}(\boldsymbol{u}(\xi))=\frac{1}{2}\left(\frac{\partial{\boldsymbol{u}_i(\xi)}} {\partial \xi_j}
+ \frac{\partial {\boldsymbol{u}_j(\xi)}}{\partial \xi_i}\right), \quad
i,j=1,\ldots,d
$$
the (small-)strain tensor. We assume that our system is governed by
linear Hooke's law, i.e., the stress is a linear function of the
strain
$$
\boldsymbol{\sigma}_{ij}(\xi) = \boldsymbol{E}_{ijk\ell}(\xi) \boldsymbol{e}_{k\ell}(\boldsymbol{u}(\xi)) \qquad(\mbox{in
	tensor notation}),
$$
where $\boldsymbol{E}$ is the elastic (plane-stress for $d=2$) stiffness tensor.

Assume that the boundary of ${\rm \Omega}$ is partitioned as $\partial{\rm \Omega}={\rm\Gamma}_u\cup{\rm\Gamma}_f$, ${\rm\Gamma}_u\cap{\rm\Gamma}_f = \emptyset$ and that an external load function ${\boldsymbol{f}}\in [L_2({\rm\Gamma}_f)]^d$ is given. Define ${\cal V}=\{\boldsymbol{u} \in [H^1({\rm \Omega})]^d \,|\, \boldsymbol{u} = 0 ~{\rm on}~{{\rm\Gamma}}_u \}\supset [H^1({\rm \Omega})]^d$. The weak form of the linear elasticity problem reads as:
\begin{align}\label{eq:ee}
&\mbox{Find } \boldsymbol{u} \in {\cal V}, \mbox{ such that }\\
&\quad \int_{{\rm \Omega}} \boldsymbol{a}(\boldsymbol{x};\boldsymbol{u},\boldsymbol{v}) =
\int_{{\rm\Gamma}_f} \boldsymbol{f}(\xi)\cdot \boldsymbol{v}(\xi)
\ds,\quad\forall \boldsymbol{v}\in {\cal V},\nonumber
\end{align}
where
\begin{equation}\label{eq:a}
	\boldsymbol{a}(\boldsymbol{x};\boldsymbol{u},\boldsymbol{v}) = \int_{{\rm \Omega}} \langle \boldsymbol{x}(\xi)\boldsymbol{E}(\xi) \boldsymbol{e}(\boldsymbol{u}(\xi)),\boldsymbol{e}(\boldsymbol{v}(\xi)) \rangle \dx\,.
\end{equation}
In the basic
\emph{topology optimization}, the design variable is the multiplier $\boldsymbol{x}\in L_\infty({\rm \Omega})$ of the
elastic stiffness tensor $\boldsymbol{E}$ which is a function of the space
variable $\xi$. We will consider the following constraints on $\boldsymbol{x}$:
$$
  \int_{\rm \Omega}\boldsymbol{x}(\xi) \dx = V,\quad
\underline{\boldsymbol{x}}\leq \boldsymbol{x}\leq \overline{\boldsymbol{x}}\ \mbox{ a.e.\ in } {\rm \Omega}
$$
with some given positive ``volume" $V$ and with $\underline{\boldsymbol{x}},\overline{\boldsymbol{x}}\in L_\infty({\rm \Omega})$ satisfying $0\leq\underline{\boldsymbol{x}}\leq\overline{\boldsymbol{x}}$ and $\int_{\rm \Omega}\underline{\boldsymbol{x}}(\xi) \dx < V < \int_{\rm \Omega}\overline{\boldsymbol{x}}(\xi) \dx$.

The choice of ${L}_\infty$ is due to the fact that we want
to allow for material/no-material situations.

The \emph{minimum compliance single-load
	topology optimization problem} reads as
\begin{align}\label{eq:VTS}
&\inf\limits_{\boldsymbol{x}\in L_\infty}
\int_{{\rm\Gamma}_f} \boldsymbol{f}(\xi)\cdot \boldsymbol{u}(\xi) \dx\\
&\mbox{subject to}\nonumber\\
&\qquad \boldsymbol{u}\mbox{~solves~(\ref{eq:ee})}\nonumber\\
&\qquad \int_{\rm \Omega} \boldsymbol{x}(\xi) \dx = V\nonumber\\
&\qquad \underline{\boldsymbol{x}} \leq \boldsymbol{x}\leq \overline{\boldsymbol{x}}\ \mbox{ a.e.\ in } {\rm \Omega} \,.\nonumber
\end{align}
The objective, the so called compliance functional,
measures how well the structure can carry the load $\boldsymbol{f}$.

Problem (\ref{eq:VTS}) is now discretized using the standard finite element method; the details can be found, e.g., in \cite{mdfmo,petersson1999finite}. In particular, we use quadrilateral elements, element-wise constant approximation of function $\boldsymbol{x}$ and element-wise bilinear approximation of the displacement field $\boldsymbol{u}$. 
After discretization, the variables will be vectors $x\in\RR^m$ and $u\in\RR^n$, where $m$ is the number of finite elements and $n$ the number of degrees of freedom (the number of finite element nodes times the spatial dimension). With every element we associate the local (symmetric and positive semidefinite) stiffness matrix $K_i$ and (for elements including part of the boundary ${\rm\Gamma}_f$) the discrete load vector $f_i$, $i=1,\ldots,m$. Now we can formulate the discretized version of the linear elasticity problem (\ref{eq:ee}) as the following system of linear equations
\begin{equation}\label{eq:eed}
K(x) u = f
\end{equation}
where
$
  K(x) = \displaystyle\sum_{i=1}^m x_i K_i
$
is the so-called global stiffness matrix and $
  f = \displaystyle\sum_{i=1}^m f_i
$
is the finite element assembly of the load vector.

The topology optimization problem (\ref{eq:VTS}) becomes
\begin{align}\label{eq:mincomp}
    & \min_{u\in\RR^n,\,x\in\RR^m,\,\gamma\in\RR\ } \gamma \\
    & \mbox{subject to} \nonumber\\
    &\qquad K(x)u = f \  \nonumber\\
    &\qquad f^\top u \leq \gamma  \nonumber\\
      &\qquad \sum_{i=1}^m x_i \leq V \nonumber\\
       &\qquad \underline{x}_i\leq x_i\leq \overline{x}_i\,, \quad i=1,\ldots,m\,. \nonumber
\end{align}

Using the Schur complement theorem, the compliance constraint and the
equilibrium equation can be written as one matrix inequality constraint:
\begin{equation}\label{eq:Z}
  Z(x) : = \begin{bmatrix}
  K(x)& f \\
  f^\top & \gamma
  \end{bmatrix} \succeq 0\,.
\end{equation}
  
The minimum compliance problem can then be formulated as follows:
\begin{align}\label{eq:SDP}
    &\min_{x\in\RR^m,\,\gamma\in\RR\ }\gamma \\
    &\mbox{subject to}\nonumber\\
      &\qquad Z(x) \succeq 0 \nonumber\\
      &\qquad \sum_{i=1}^m x_i \leq V \nonumber\\
      &\qquad\underline{x}_i\leq x_i\leq \overline{x}_i\,, \quad i=1,\ldots,m \nonumber\,.
\end{align}

For ease of notation, in the rest of the paper we will restrict ourselves to the planar case $d=2$. Generalization of all ideas to the three-dimensional case is straightforward.

\section{Decomposition of the topology optimization problem (\ref{eq:SDP})}\label{sec:topodec}
Let ${\rm \Omega}_h\subset\RR^2$ be a polygonal approximation of ${\rm \Omega}$ discretized by finite elements. Assume that ${\rm \Omega}_h$ is partitioned into $p$ non-overlapping subdomains $D_k$, $k=1,\ldots,p$, whose boundaries coincide with finite element boundaries. In our examples ${\rm \Omega}={\rm \Omega}_h$ is a rectangle, the underlying finite element mesh is regular and so is the partitioning into the subdomains. Confront Figure~\ref{fig:mesh} that shows typical decomposition of ${\rm \Omega}_h$ into $N_x\times N_y$ subdomains.
\begin{figure}[h]
	\begin{center}
		\resizebox{0.5\hsize}{!}
		{\includegraphics{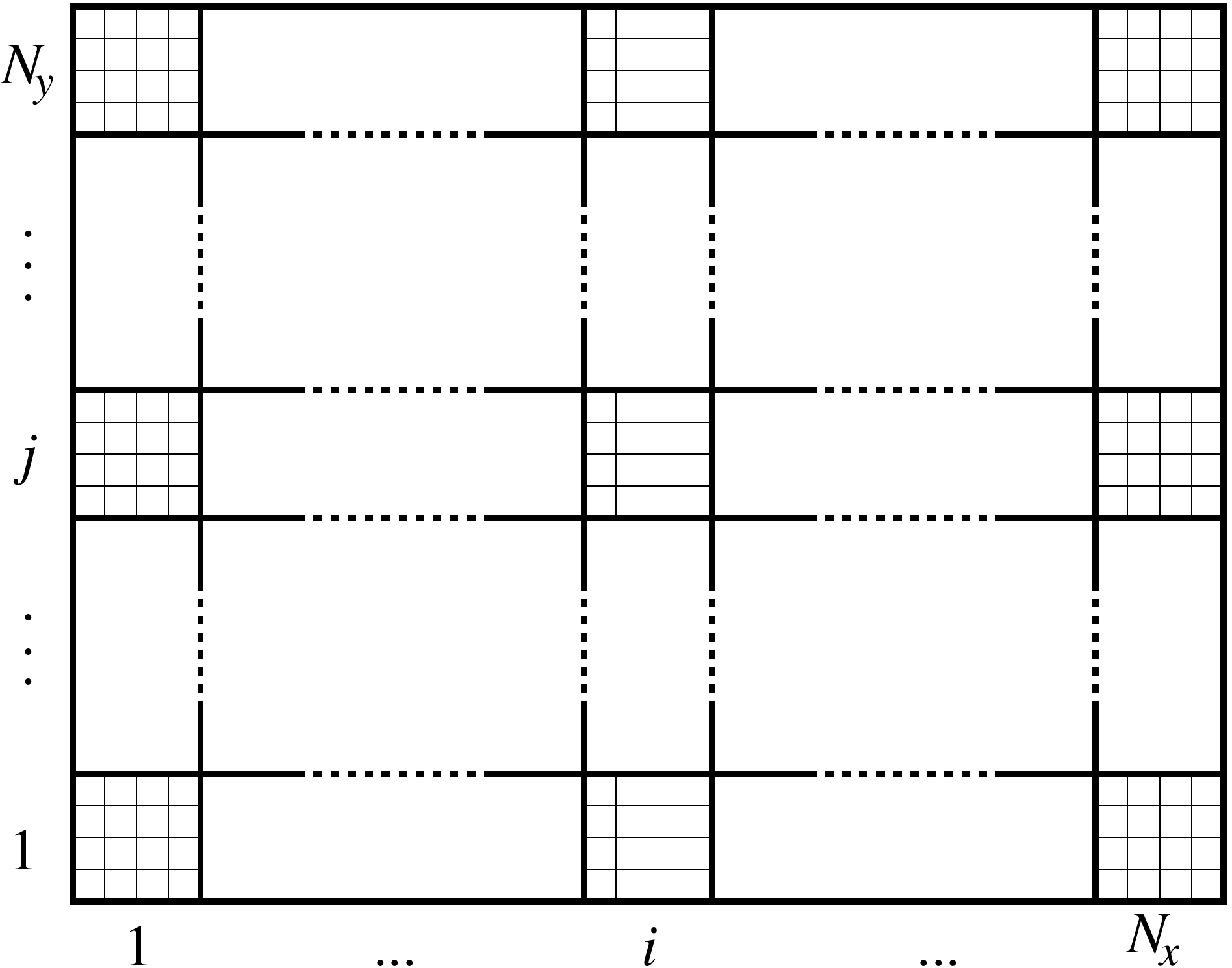}}
	\end{center}
	\caption{Regular partitioning of the computational domain into subdomains coinciding with groups of finite elements.}\label{fig:mesh}
\end{figure}

Let $I_k$ be the index set of all degrees of freedom associated with the subdomain $D_k$, $k=1,\ldots,p$. The intersections of these index sets will include the degrees of freedom on the respective internal boundaries and will be again denoted by
$$
  I_{k,\ell} = I_k\cap I_\ell, \quad (k,\ell)\in\Theta_p\,.
$$

Denote by ${\cal D}_{k}$ the index set of elements belonging to subdomain $D_k$ and define
\begin{equation}\label{eq:hatK}
{K}^{(k)}(x)= \sum_{i\in {\cal D}_{k}}x_i\, K_i\,.
\end{equation}
Matrix ${K}^{(k)}(x) = K^{(k)}$ can then be partitioned as follows
\begin{equation*}
K^{(k)} =
\begin{bmatrix}
{ K^{(k)}_{{\cal I}{\cal I}}} &{ K^{(k)}_{{\cal I}\Gamma}}  \\[.5em]
{ K^{(k)}_{\Gamma {\cal I}}}& { K^{(k)}_{\Gamma\Gamma}}
\end{bmatrix}
\end{equation*}
where the set $\Gamma$ collects indices of all degrees of freedom corresponding with indices in one of he sets $I_{\ell,k}$ or $I_{k,\ell}$, $\ell=1,\ldots,p$; the set ${\cal I}$ then collects indices of all remaining ``interior" degrees of freedom in ${D}_{k}$.

We are now in a position to apply the theorems from Section~\ref{sec:deco}. 

\paragraph{Case A -- Chordal decomposition} Let us first apply Corollary~\ref{th:coro}. It says that the matrix inequality $Z(x)\succeq 0$ from (\ref{eq:SDP}) can be equivalently replaced by the following matrix inequalities 
\begin{equation}\label{eq:casea}
Z^{(k)}_A:=\begin{bmatrix}
{ K}^{(k)}_{{\cal I}{\cal I}}(x) &{ K}^{(k)}_{{\cal I}\Gamma}(x) & 0 \\[.5em]
{ K}^{(k)}_{\Gamma {\cal I}}(x)& { K}^{(k)}_{\Gamma\Gamma}(x) & f^{(k)} \\[.5em]
0 & (f^{(k)})^{\top}& 0
\end{bmatrix} 
+
\begin{bmatrix}
0  & 0 & 0 \\[.5em]
0 & S^{(k)} & \sigma^{(k)} \\[.5em]
0 & (\sigma^{(k)})^{\top}& { s^{(k)}}
\end{bmatrix}
\succeq 0
\end{equation} 
where
\begin{align}
  S^{(k)} &= 
  - \sum_{\substack{\ell:\ell<k\\I_{\ell,k}\not=\emptyset}}S_{\ell,k} + \sum_{\substack{\ell:\ell>k\\I_{k,\ell}\not=\emptyset}}S_{k,\ell}\label{eq:casea1}\\
  \sigma^{(k)} &= 
  - \sum_{\substack{\ell:\ell<k\\I_{\ell,k}\not=\emptyset}}\sigma_{\ell,k} + \sum_{\substack{\ell:\ell>k\\I_{k,\ell}\not=\emptyset}}\sigma_{k,\ell}
\,.\label{eq:casea2}
\end{align}
The additional variables are the matrices, vectors and scalars 
$$
  S_{k,\ell}\in\SS^{|I_{k,\ell}|},\quad \sigma_{k,\ell}\in\RR^{|I_{k,\ell}|},\quad
  s\in\RR^p,\quad (k,\ell)\in\Theta_p\,.
$$
\paragraph{Case B -- Arrow decomposition} Now we apply Theorem~\ref{th:theo}. In this case, the matrix inequality $Z(x)\succeq 0$ from (\ref{eq:SDP}) can be replaced by the following matrix inequalities 
\begin{equation}\label{eq:caseb}
Z^{(k)}_B:=\begin{bmatrix}
{ K}^{(k)}_{{\cal I}{\cal I}}(x) &{ K}^{(k)}_{{\cal I}\Gamma}(x) & 0 \\[.5em]
{ K}^{(k)}_{\Gamma {\cal I}}(x)& { K}^{(k)}_{\Gamma\Gamma}(x) &{f^{(k)}} \\[.5em]
0 & (f^{(k)})^{\top}& 0
\end{bmatrix} 
+
\begin{bmatrix}
0  & 0 & 0 \\[.5em]
0 & 0 & g^{(k)} \\[.5em]
0 & (g^{(k)})^{\top}& \gamma^{(k)}
\end{bmatrix}
\succeq 0
\end{equation}
where
\begin{align}
g^{(k)} &= - \sum_{\substack{\ell:\ell<k\\I_{\ell,k}\not=\emptyset}}g_{\ell,k} + \sum_{\substack{\ell:\ell>k\\I_{k,\ell}\not=\emptyset}}g_{k,\ell}\,.\label{eq:caseb1}
\end{align}
The additional variables $g_{\bullet,\bullet}$ and $\gamma$, respectively, have the same dimensions as the variables $\sigma_{\bullet,\bullet}$ and $s$ in Case A.

Recall that Theorem~\ref{th:theo} does not use the restrictive Assumption~4 from Section~\ref{sec:deco}. This is important, because Assumption~4 is not satisfied when the domain $\rm\Omega$ contains holes, and so the decomposition technique would not be applicable to some practical problems. Consider, for instance, the finite element mesh in Figure~\ref{fig:mesh} and assume that the $(i,j)$th subdomain is not part of the domain $\rm\Omega$, it is a hole with no finite elements. Then, even if we assume all matrices $K^{(k)}$ to be dense, the sparsity graph of $K(x)$ is not chordal, as it contains the chordless cycle connecting (more than 3) nodes on the boundary of the internal hole.

Before formulating the decomposed version of problem (\ref{eq:SDP}) we notice that, according to Corollary~\ref{th:coro} and Theorem~\ref{th:theo}, $Z = \sum_{k=1}^p Z^{(k)}_A = \sum_{k=1}^p Z^{(k)}_B$, which means, in particular, that
$$
  \gamma = \sum_{k=1}^p s_k = \sum_{k=1}^p \gamma_k \,.
$$
We will therefore replace the variable $\gamma$ in the decomposed problems by either $s_k$ or $\gamma_k$ and the objective function by one of the above sums.

\paragraph{Case A}
Using the chordal decomposition approach, the decomposed optimization problem in variables
\begin{align*}
& x\in\RR^m,\quad
 s\in\RR^p,\\
& \sigma = \left\{  \sigma_{k,\ell}\right\}_{(k,\ell)\in\Theta_p},\quad
\sigma_{k,\ell}\in\RR^{|I_{k,\ell}|}\\
& S = \left\{ S_{k,\ell}\right\}_{(k,\ell)\in\Theta_p},\quad S_{k,\ell}\in\SS^{|I_{k,\ell}|}
\end{align*}
is formulated as follows
\begin{align}
& \min_{x,\,s,\,\sigma,\,S\ }
\sum_{k=1}^p s_k\label{eq:casea_dec}\\
&\mbox{subject to}\nonumber\\
&\quad \sum_{i\in {\cal D}} x_i\leq V\nonumber\\
&\quad \underline{x}\leq x\leq\overline{x}\nonumber\\
&\quad Z_A^{(k)} 
\succeq 0\qquad \displaystyle k=1,\ldots,p \nonumber
\end{align}		
with $Z_A^{(k)}$ defined as in (\ref{eq:casea}),(\ref{eq:casea1}),(\ref{eq:casea2}). 

\paragraph{Case B}
Using the arrow decomposition approach, the decomposed optimization problem in variables
\begin{align*}
& x\in\RR^m,\quad
 \gamma\in\RR^p,\\
& g = \left\{  g_{k,\ell}\right\}_{(k,\ell)\in\Theta_p},\quad
g_{k,\ell}\in\RR^{|I_{k,\ell}|}
\end{align*}
reads as
\begin{align}
& \min_{x,\,\gamma,\,g\ }
\sum_{k=1}^p\gamma_k \label{eq:caseb_dec}\\
&\mbox{subject to}\nonumber\\
&\quad \sum_{i\in {\cal D}} x_i\leq V\nonumber\\
&\quad \underline{x}\leq x\leq\overline{x}\nonumber\\
&\quad Z_B^{(k)} 
\succeq 0\qquad \displaystyle k=1,\ldots,p \nonumber
\end{align}		
with $Z_B^{(i,j)}$ defined as in (\ref{eq:caseb}),(\ref{eq:caseb1}). 

\paragraph{A versus B}
Consider now the finite element mesh and decomposition as in Figure~\ref{fig:mesh} with $n_x\times n_y$ finite elements and $N_x\times N_y$ subdomains.
Instead of (\ref{eq:SDP}) we can solve one of the decomposed problems (\ref{eq:casea_dec}) and (\ref{eq:caseb_dec}). In Case A of the chordal decomposition the single matrix inequality of dimension $(n+1)\times(n+1)$ is replaced by $N_x\cdot N_y$ inequalities of dimension of order $2(n_x/N_x+1)(n_y/N_y+1)+1$ while we have to add $N_x(N_y-1)+(N_x-1)N_y$ additional vectors $\sigma_{\bullet,\bullet}$ of a typical size $2(n_x/N_x+1)$ or $2(n_y/N_y+1)$, the same number of additional (dense) matrix variables $S_{\bullet,\bullet}$ of the same order and $N_x\cdot N_y$ scalar variables $s_{\bullet}$. (Recall that the factor 2 stems from the fact that there are two degrees of freedom at every finite element node.)
In Case B of the arrow decomposition, the number and order of the new matrix constraints is the same as above but we only need the additional scalar and vector variables; the additional matrix variables are not necessary.

Later in Section~\ref{sec:numer} we will see that this decomposition leads to enormous speed-up in computational time of a state-of-the-art SDO solver. We will also see that the omission of the additional matrix variables in the arrow decomposition can make a big difference.

\paragraph{Example 1} The notation used in the above decomposition approaches is rather cumbersome, so let us illustrate it using a simple example.
\begin{figure}[h]
	\begin{center}
		\resizebox{0.7\hsize}{!}
		{\includegraphics{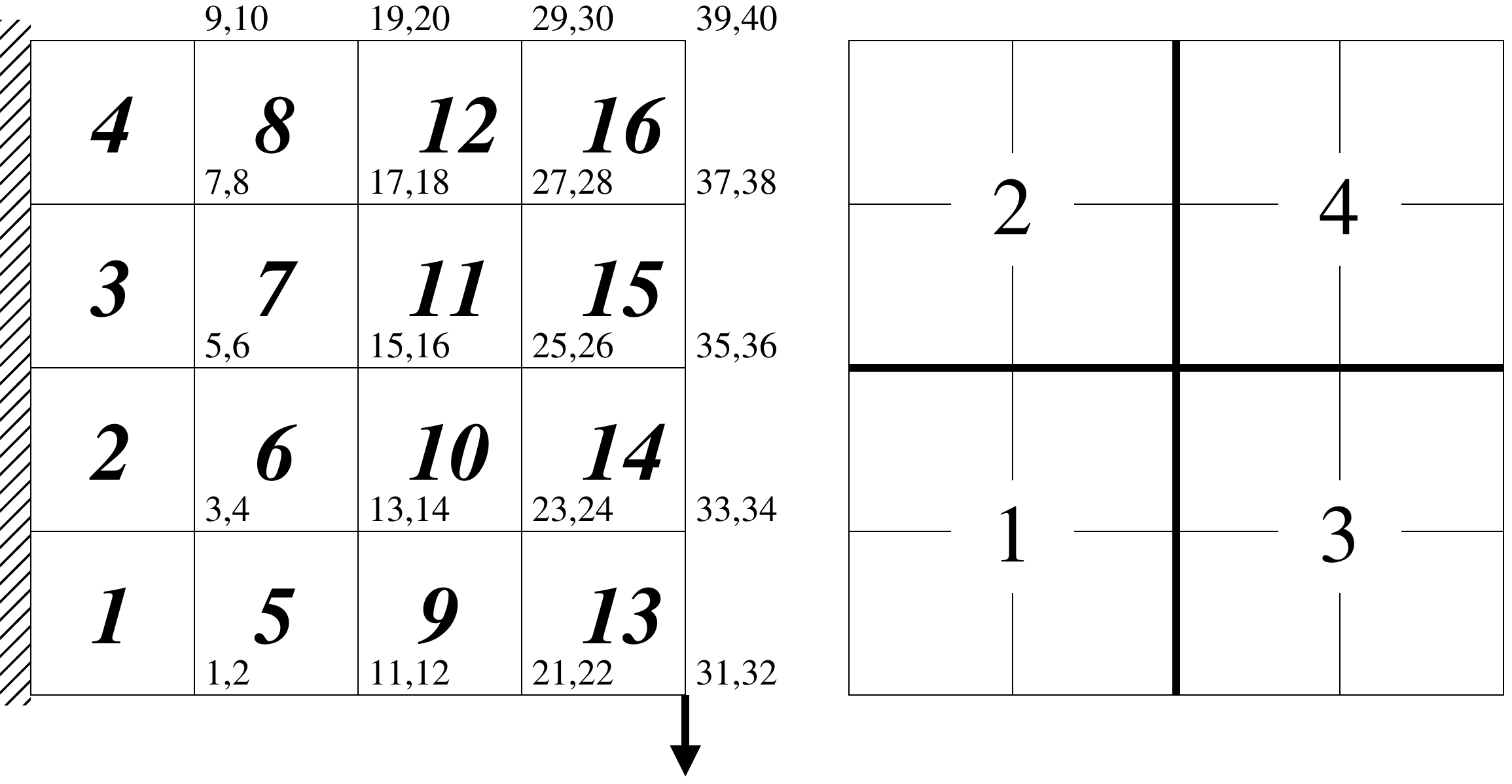}}
	\end{center}
	\caption{Example 1: Problem setting, finite element mesh (left) and decomposition into four subdomains (right).}\label{fig:ex_mesh}
\end{figure}
Figure~\ref{fig:ex_mesh} presents a finite element mesh with 16 elements and 25 nodes. All nodes on the left-hand side are fixed and thus eliminated from the stiffness matrix. Hence the corresponding stiffness matrix will have dimension 40$\times$40 (two degrees of freedom associated with every free finite element node, as depicted in the figure). The structure of the corresponding stiffness matrix $K$ is shown in Figure~\ref{fig:mat_glob}; here the elements corresponding to interior degrees of freedom (index sets $\cal I$) are denoted by circles, while elements associated with the the intersections $I_{k,\ell}$ are marked by full dots. 
\begin{figure}[h]
	\begin{center}
		\resizebox{0.4\hsize}{!}
		{\includegraphics{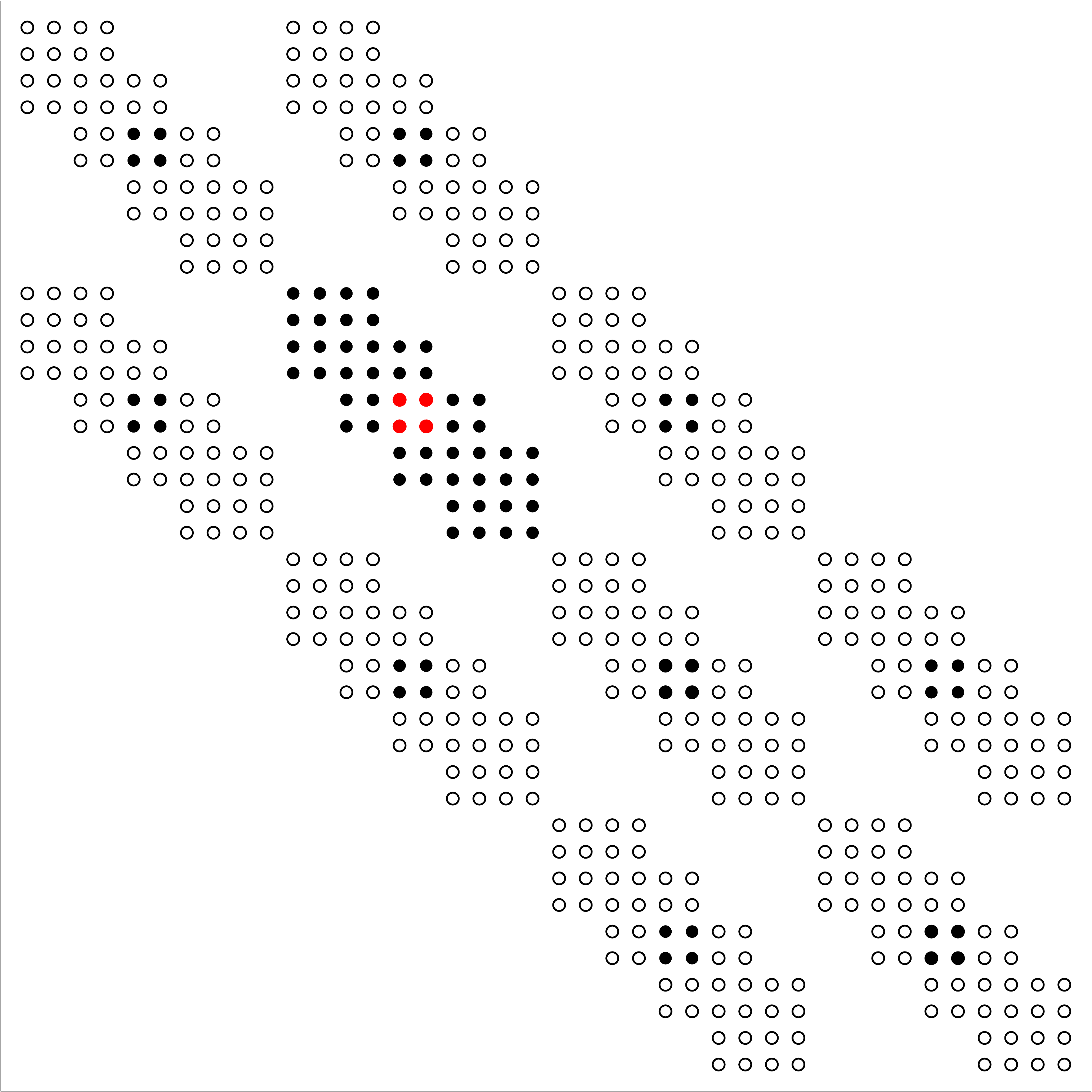}}
	\end{center}
	\caption{Sparsity structure of stiffness matrix $K$ in Example 1.}\label{fig:mat_glob}
\end{figure}

Thus in the original topology optimization problem (\ref{eq:SDP}) we have $n=40$ and $m=16$ and the matrix constraint $Z(x)\succeq 0$ is of dimension $41\times 41$. We now decompose the problem into four subdomains, containing elements $\{1,2,5,6\}$, $\{3,4,7,8\}$, $\{9,10,13,14\}$, $\{11,12,15,16\}$; see Figure~\ref{fig:ex_mesh}--right. Then
\begin{align*}
I_1 &= \{1,\ldots,6,11,\ldots,16\},\quad
I_2 = \{21,\ldots,26,31,\ldots,36\},\\
I_3 &= \{5,\ldots,10,15,\ldots,20\},\quad
I_4 = \{25,\ldots,30,35,\ldots,40\},\\
I_{1,2} &= \{5,6,15,16\},\quad I_{1,3} = \{11,\ldots,16\},\quad
I_{1,4} = \{15,16\},\\ I_{2,3} &= \{15,16\},\quad
I_{2,4}  = \{15,\ldots,20\},\quad I_{3,4} = \{15,16,25,26,35,36\}.
\end{align*}
The structure of the stiffness matrices associated with domains 1--4 is shown, left-to-right, in Figure~\ref{fig:mat14}.
Notice that indices 15,16 (marked by red dots in Figures~\ref{fig:mat_glob},\ref{fig:mat14}) are contained in all six sets $I_{\bullet,\bullet}$.
\begin{figure}[h]
	\begin{center}
		\resizebox{0.24\hsize}{!}
		{\includegraphics{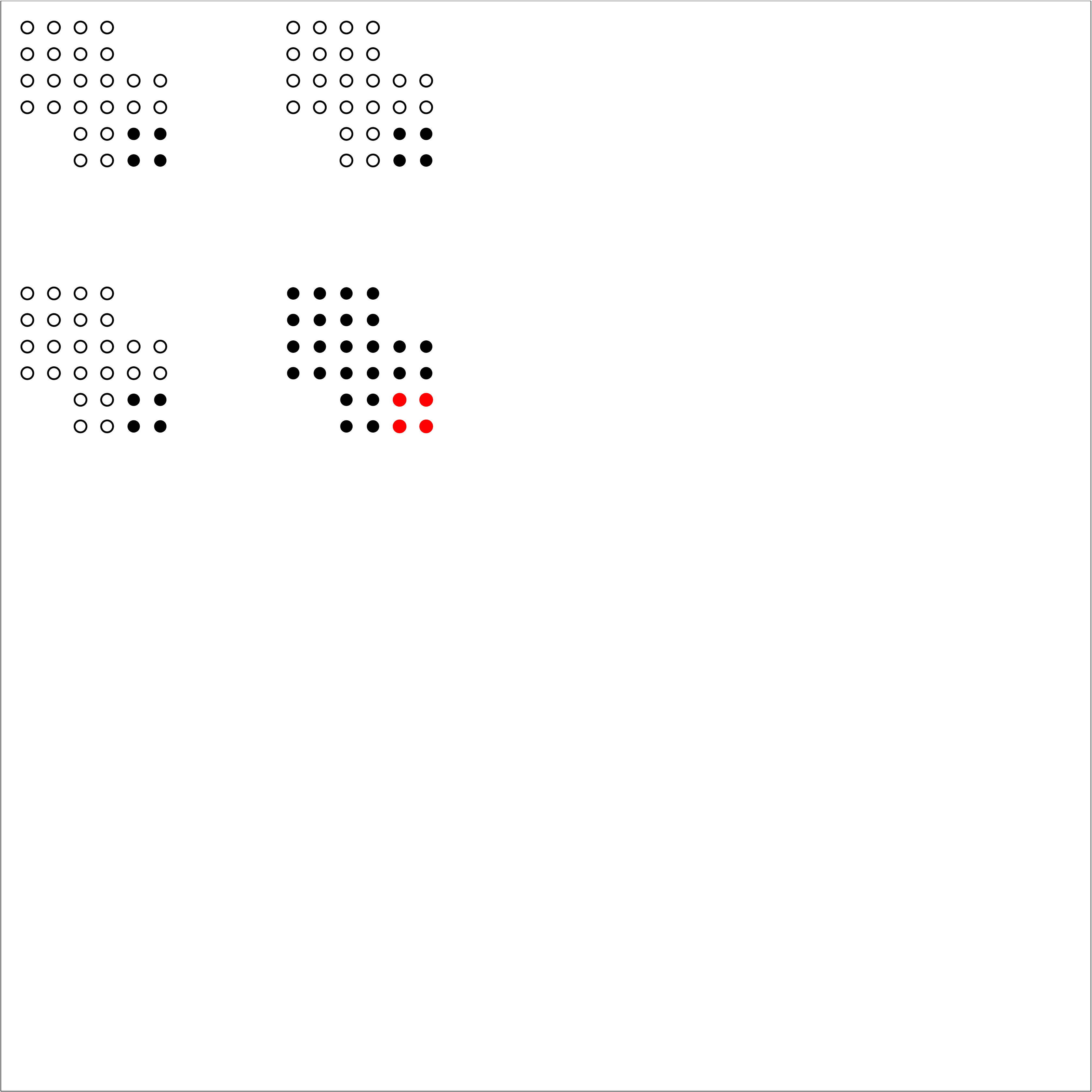}}\ 
		\resizebox{0.24\hsize}{!}
		{\includegraphics{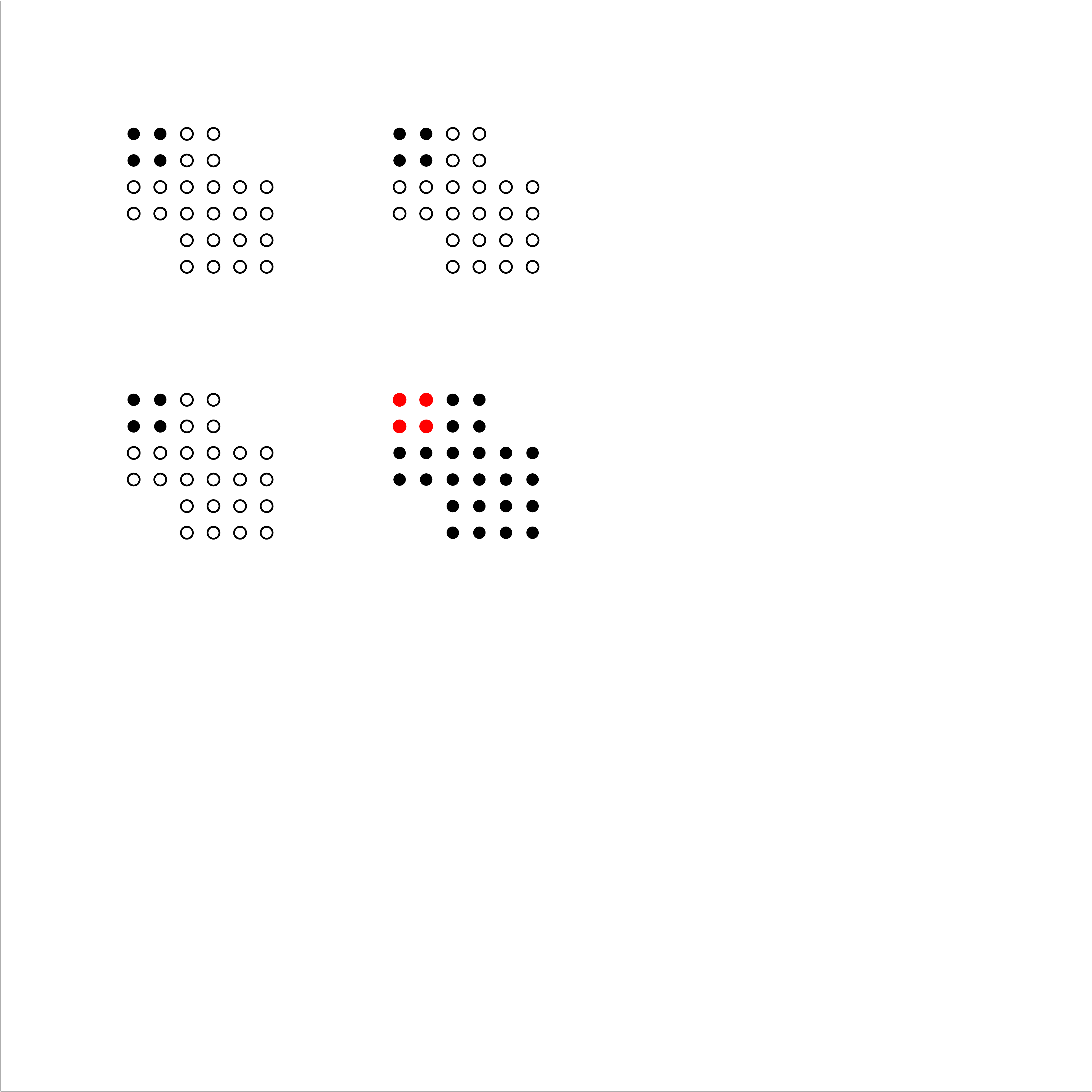}}\ 
		\resizebox{0.24\hsize}{!}
		{\includegraphics{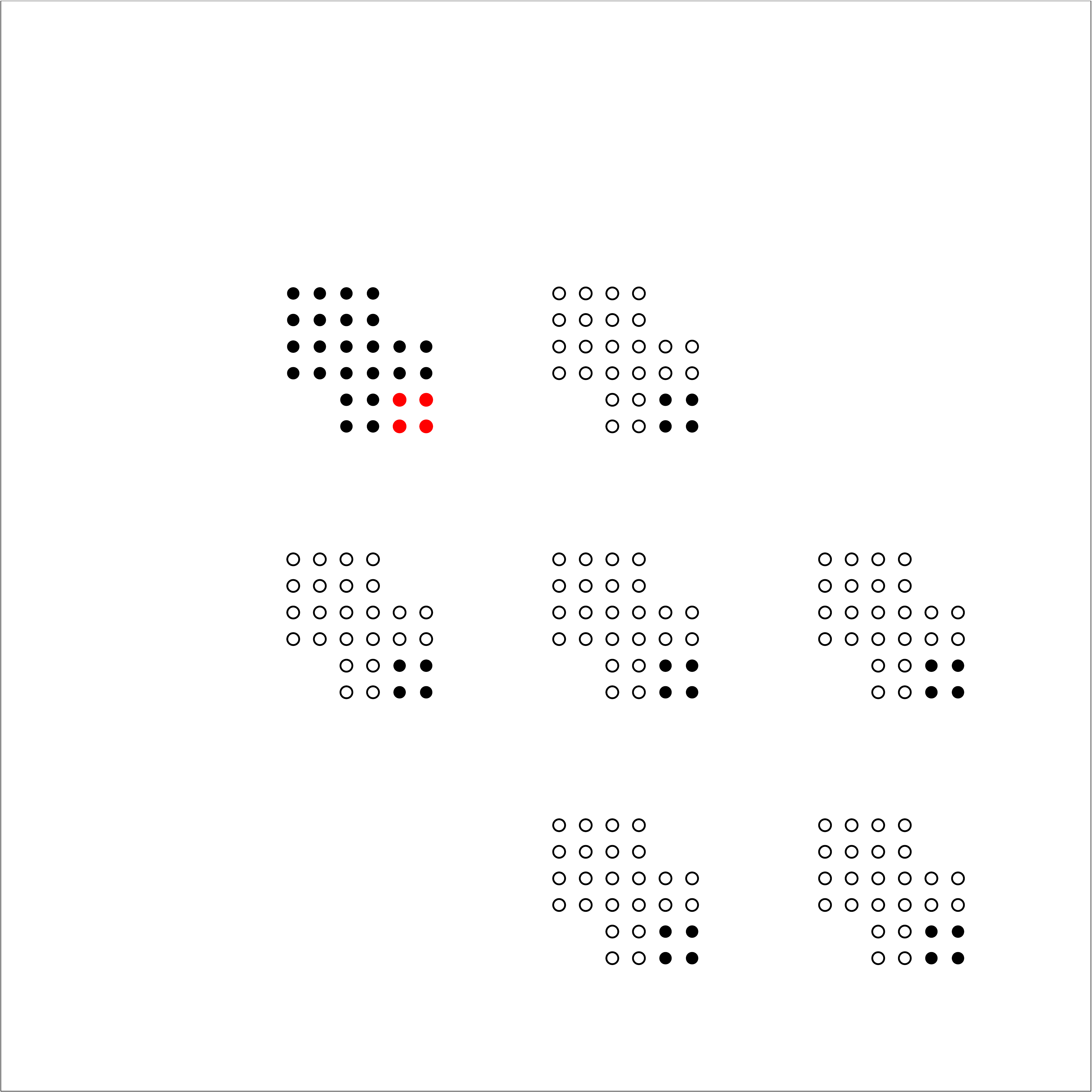}}\ 
		\resizebox{0.24\hsize}{!}
		{\includegraphics{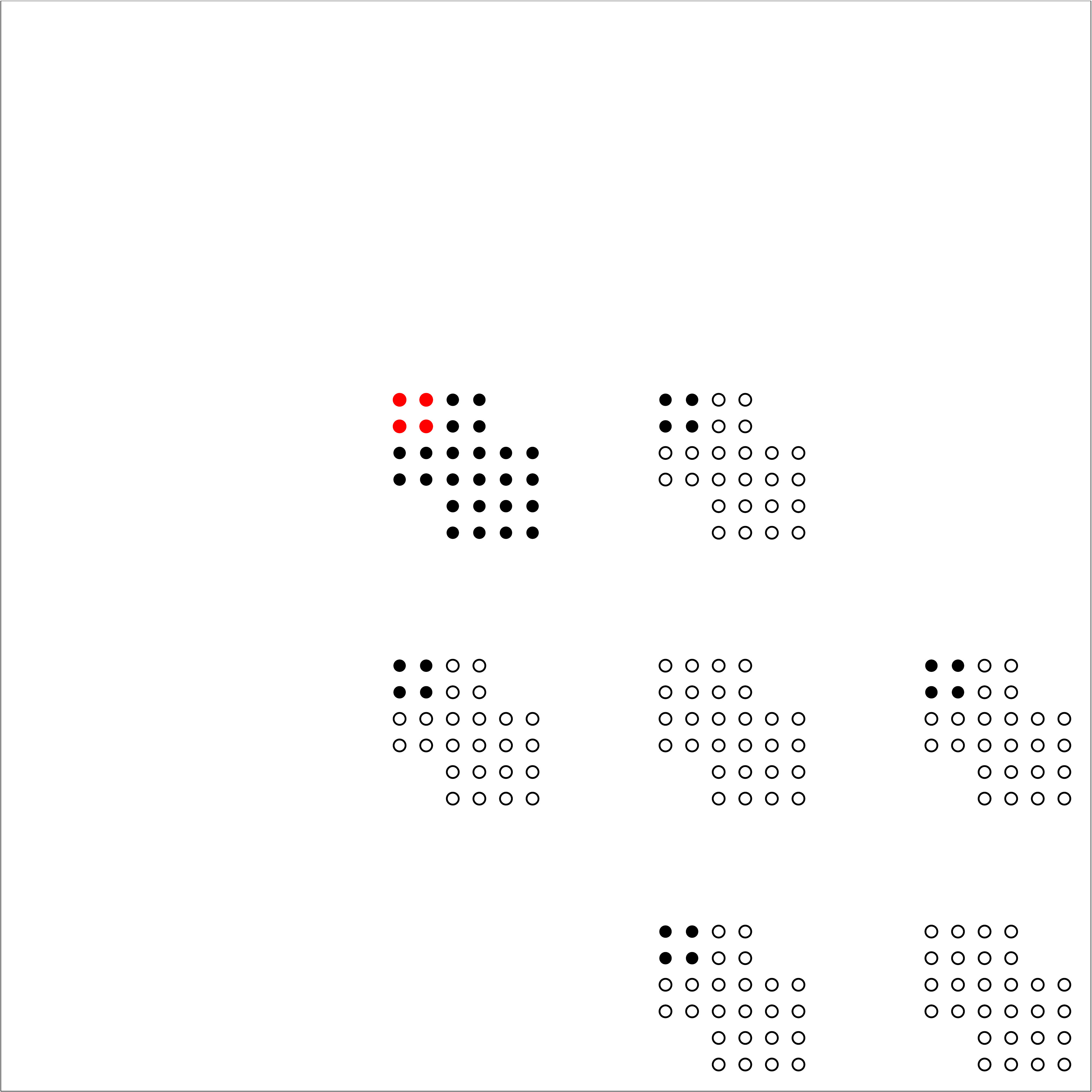}}
	\end{center}
	\caption{Sparsity structure of stiffness matrices $K_1,\ldots,K_4$ associated with subdomains 1--4 in Example 1.}\label{fig:mat14}
\end{figure}

The chordal decomposition problem (\ref{eq:casea_dec}) will have four matrix constraints, two of order 13 and two of order 19, and additional variables $s\in\RR^6$,  $\sigma_{1,4}, \sigma_{2,3}\in\RR^2$, $\sigma_{1,2}\in\RR^4$, $\sigma_{1,3},\sigma_{2,4},\sigma_{3,4}\in\RR^6$ and 
$S_{1,4}, S_{2,3}\in\SS^2$, $S_{1,2}\in\SS^4$, $S_{1,3},S_{2,4},S_{3,4}\in\SS^6$.
The arrow decomposition problem (\ref{eq:caseb_dec}) will have the same number of matrix constraints as (\ref{eq:casea_dec}) and additional variables $\gamma\in\RR^6$, $g_{1,4}, g_{2,3}\in\RR^2$, $g_{1,2}\in\RR^4$, $g_{1,3},g_{2,4},g_{3,4}\in\RR^6$.

\section{Decomposition by fictitious loads}
So far, all the reasoning was purely algebraic. There is, however, an alternative, functional analytic view of the arrow decomposition in Theorem~\ref{th:theo}. We will present it in this section. The purpose is to illustrate a different viewpoint and so, to keep the notation simple, we will only consider the case of two subdomains.
%
%

\subsection{Infinite dimensional setting}
Let us recall the weak formulation (\ref{eq:ee}) of the elasticity problem depending on parameter~$x$:
\begin{align}\label{eq:weak}
\boldsymbol{a}(\boldsymbol{x};\boldsymbol{u},\boldsymbol{v}) = \int_{{\rm\Gamma}_f} \boldsymbol{f}\boldsymbol{v}\; {\rm d}s \quad\forall \boldsymbol{v}\in {\cal V}\,.
\end{align}

Let ${\rm \Omega}$ be partitioned into two mutually disjoint subdomains ${\rm \Omega}_1$ and ${\rm \Omega}_2$ such that ${\rm \Omega}_1\cup{\rm \Omega}_2={\rm \Omega}$. Denote the interface boundary between the two subdomains by ${\rm\Gamma}_I$; see Figure~\ref{fig:interf}. We consider the general situation when ${\rm\Gamma}_u$ and ${\rm\Gamma}_f$ may be a part of both, $\partial{\rm \Omega}_1\cap\partial\Omega$ and $\partial{\rm \Omega}_2\cap\partial\Omega$. Define $\boldsymbol{a}_i$ as a restriction of the bilinear form $\boldsymbol{a}$ to ${\rm \Omega}_i$ (the integral in (\ref{eq:a}) is simply computed over ${\rm \Omega}_i$), $\boldsymbol{f}_i=\boldsymbol{f}|_{\partial{\rm \Omega}_i}$ and
$$
  {\cal V}_i=\{\boldsymbol{v}\in [H^1({\rm \Omega}_i)]^2 \mid \boldsymbol{v}=0\ \mbox{on}\ ({\rm\Gamma}_u\cap\partial{\rm \Omega}_i)\cup{\rm\Gamma}_I\},\ \  i=1,2\,.
$$
\begin{figure}[h]
	\begin{center}
		\resizebox{0.4\hsize}{!}
		{\includegraphics{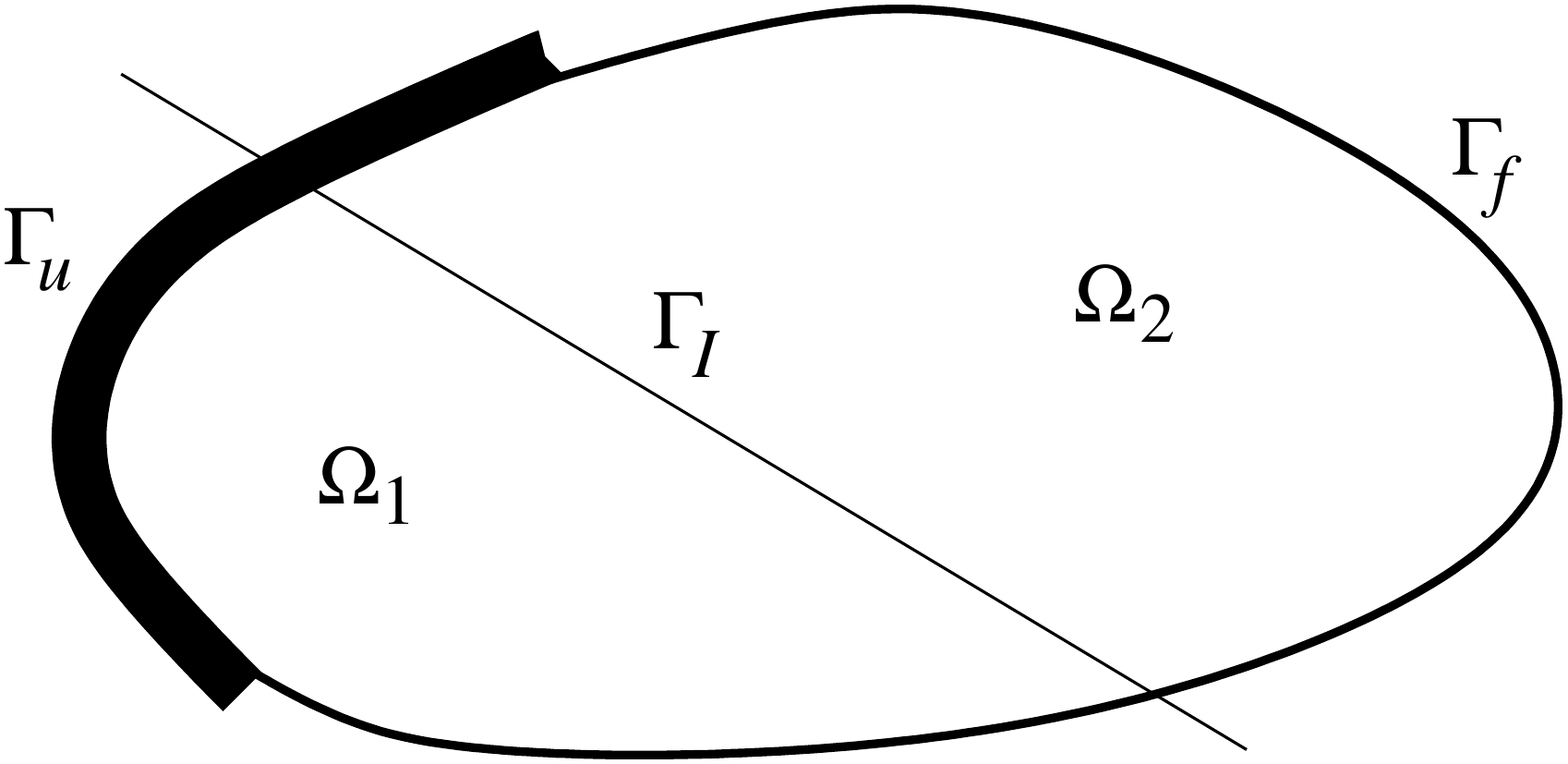}}
	\end{center}
	\caption{Partitioning of domain ${\rm \Omega}$ into two subdomains with interface boundary ${\rm\Gamma}_I$.}\label{fig:interf}
\end{figure}

Consider the following ``restricted'' problems:
	\begin{align}
	\mbox{Find $\boldsymbol{u}\in [H^1({\rm \Omega}_1)]^2$ such that}\quad &\boldsymbol{u}-\boldsymbol{u}^*\in {\cal V}_1\mbox{\ and}\label{eq:weak1s}\\
	&\boldsymbol{a}_1(\boldsymbol{x};\boldsymbol{u},\boldsymbol{v}) = \int_{{\rm\Gamma}_f\cap\partial{\rm \Omega}_1}  \boldsymbol{f}_1\boldsymbol{v} \; {\rm d}s \quad\forall \boldsymbol{v}\in {\cal V}_1\,;\nonumber\\
	\mbox{Find $\boldsymbol{u}\in [H^1({\rm \Omega}_2)]^2$ such that}\quad &\boldsymbol{u}-\boldsymbol{u}^*\in {\cal V}_2\mbox{\ and} \label{eq:weak2s} \\
	&\boldsymbol{a}_2(\boldsymbol{x};\boldsymbol{u},\boldsymbol{v}) = \int_{{\rm\Gamma}_f\cap\partial{\rm \Omega}_2}  \boldsymbol{f}_2\boldsymbol{v} \; {\rm d}s \quad\forall \boldsymbol{v}\in {\cal V}_2\,.\nonumber
	\end{align}

The following theorem forms a basis of our approach.
\begin{theorem}\label{th:1}
	Assume that $\boldsymbol{u}^*$ solves (\ref{eq:weak}). For all $\boldsymbol{x}\in L_\infty({\rm \Omega})$ there exists $\boldsymbol{g}\in [H^{-1/2}({\rm\Gamma}_I)]^2$ such that solutions to (\ref{eq:weak1s}) and (\ref{eq:weak2s}) are equal to respective solutions of the following problems
	\begin{align}
	\boldsymbol{a}_1(\boldsymbol{x};\boldsymbol{u},\boldsymbol{v}) &=  \int_{{\rm\Gamma}_f\cap\partial{\rm \Omega}_1}  \boldsymbol{f}_1\boldsymbol{v} \; {\rm d}s + \langle \boldsymbol{g},\boldsymbol{v}\rangle_{{\rm\Gamma}_I} \quad\forall \boldsymbol{v}\in {\cal V}({\rm \Omega}_1)\label{eq:weak1}\\
	\boldsymbol{a}_2(\boldsymbol{x};\boldsymbol{u},\boldsymbol{v}) &= \int_{{\rm\Gamma}_f\cap\partial{\rm \Omega}_2}  \boldsymbol{f}_2\boldsymbol{v} \; {\rm d}s - \langle \boldsymbol{g},\boldsymbol{v}\rangle_{{\rm\Gamma}_I} \quad\forall \boldsymbol{v}\in {\cal V}({\rm \Omega}_2)\,,\label{eq:weak2}
	\end{align}
	where $\langle \cdot,\cdot\rangle_{{\rm\Gamma}_I}$ denotes the duality pairing between $[H^{-1/2}({\rm\Gamma}_I)]^2$ and  $[H^{1/2}({\rm\Gamma}_I)]^2$.
\end{theorem}
\begin{proof}
The requested function $\boldsymbol{g}$ is the outcome of the respective Steklov-Poincar\'e operator applied to $\boldsymbol{u}^*$; see, e.g., \cite{Quarteroni1991}.\qed
\end{proof}
In the above theorem, function $g$ can be interpreted as a fictitious load applied to either of the problems (\ref{eq:weak1}),(\ref{eq:weak2}). The theorem says that there exists such a $g$ that the solutions of (\ref{eq:weak1}),(\ref{eq:weak2}) are equivalent to the solution of the ``full" problem (\ref{eq:weak}) restricted to the respective subdomain. Or, in other words, the solutions of (\ref{eq:weak1}),(\ref{eq:weak2}) can be ``glued" to form the solution of (\ref{eq:weak}).

\subsection{Finite dimensional setting}
Now assume that the discretization of ${\rm \Omega}$ is such that the interface boundary ${\rm\Gamma}_I$ is a union of boundaries of some finite elements. More precisely, we assume that the index set of finite elements used to the discretization of ${\rm \Omega}$ can be split into two disjoint subsets
$$
  \{1,2,\ldots,m\} ={\cal D}_1\cup {\cal D}_2,\quad {\cal D}_1\cap {\cal D}_2=\emptyset,
$$
such that ${\rm \Omega}_i$ is discretized by elements with indices from ${\cal D}_i$, $i=1,2$. Define
$$
f^{(1)} = \sum_{i\in{\cal D}_1} f_i,\qquad  f^{(2)} = \sum_{i\in{\cal D}_2} f_i\,,
$$
the restrictions of the load vector $f$ on boundaries of ${\rm \Omega}_1$ and ${\rm \Omega}_2$, respectively.

Denote the index set of degrees of freedom associated with finite element nodes on ${\rm\Gamma}_I$ by $I_{1,2}$. Let $n_{\scriptscriptstyle\rm\Gamma}$ be the dimension of $I_{1,2}$. 

Finally, for a vector in $z\in\RR^{n_{\scriptscriptstyle\rm\Gamma}}$ denote by $\overleftrightarrow{z}$ its extension to $\RR^n$: 
$$
  \overleftrightarrow{z}_{\!\!\!i}:=\left< \begin{aligned}&z_i \mbox{~if~}i\in I_{1,2}\\
 & 0 \mbox{~if~}i\in \RR^n\setminus I_{1,2}\end{aligned}\right.\,.
$$

The discrete version of Theorem~\ref{th:1} can then be formulated as follows. (The following corollary is, in fact, trivial in the finite dimension; however, we need the above theorem to understand the meaning of the fictitious load and its existence in the original setting of the problem.)
\begin{corollary}\label{th:col1}
Assume that $u^*$ solves (\ref{eq:eed}). Then for all $x\in \RR^m$ there exists $g\in\RR^{n_{\scriptscriptstyle\rm\Gamma}}$ such that
\begin{align}
  (\sum_{i\in{\cal D}_1}x_i K_i) u^* &= f^{(1)}+\overleftrightarrow{g} \label{eq:equi1}\\
  (\sum_{i\in{\cal D}_2}x_i K_i) u^* &= f^{(2)} - \overleftrightarrow{g}\,. \label{eq:equi2}
\end{align}
\end{corollary}
Notice that (\ref{eq:equi1}), (\ref{eq:equi2}) are still systems of dimension $n$; however, many rows and columns in the matrix and the right hand side are equal to zero, so they can be solved as systems of dimensions $|{\cal N}^{(1)}|$ and $|{\cal N}^{(2)}|$, respectively.
Hence, if we knew the fictitious load $g$, we could replace the large system of equations (\ref{eq:eed}) by two smaller ones which, numerically, would be more efficient. Of course, we do not know it. However, and this is the key idea of this section, the linear system (\ref{eq:eed}) is a constraint in an optimization problem, hence we can add $g$ among the variables and, instead of searching for the optimal design $x$ and the corresponding $u$ satisfying (\ref{eq:eed}), search for optimal $x$ and for a pair $(u,g)$ satisfying two smaller equilibrium equations (\ref{eq:equi1}) and (\ref{eq:equi2}).

We can now formulate a result regarding the decomposition of the discretized topology optimization problem (\ref{eq:mincomp}).
\begin{theorem}
Problem (\ref{eq:mincomp}) is equivalent to the following problem:
\begin{align}
	& \min_{x\in\RR^m,\,u\in\RR^n,\gamma_1\in\RR,\,\gamma_2\in\RR,\,g\in\RR^{n_{\scriptscriptstyle\rm\Gamma}}\ }
	\gamma_1 + \gamma_2\label{eq:mincomp2}\\
	&\quad \mbox{subject to}\nonumber\\
	&\quad \sum_{i=1}^m x_i\leq V\nonumber\\
	&\quad \underline{x}\leq x\leq\overline{x}\nonumber\\
	&\quad (\sum_{i\in{\cal D}_1}x_i K_i)u = f^{(1)}+\overleftrightarrow{g}\nonumber\\
	&\quad (\sum_{i\in{\cal D}_2}x_i K_i)u = f^{(2)} - \overleftrightarrow{g}\nonumber\\
	&\quad (f^{(1)}+\overleftrightarrow{g})^\top u\leq \gamma_1 \nonumber\\
	&\quad (f^{(2)}-\overleftrightarrow{g})^\top  u \leq \gamma_2 \,.\nonumber
\end{align}
 In particular, if $(\tilde{x},\tilde{u},\tilde{\gamma})$ is a solution of (\ref{eq:mincomp}) then there is $\tilde{\gamma}_1\in\RR_+, \tilde{\gamma}_2\in\RR_+,\tilde{g}\in\RR^{n_{\scriptscriptstyle\rm\Gamma}}$ such that $\tilde{\gamma}=\tilde{\gamma}_1+\tilde{\gamma}_2$ and $(\tilde{x},\tilde{u},\tilde{\gamma}_1,\tilde{\gamma}_2,\tilde{g})$ is a solution of (\ref{eq:mincomp2}). Vice versa, if
 $(\hat{x},\hat{u},\hat{\gamma}_1,\hat{\gamma}_2,\hat{g})$ is a solution of (\ref{eq:mincomp2}) then $(\hat{x},\hat{u},\hat{\gamma}_1+\hat{\gamma}_2)$ is a solution of (\ref{eq:mincomp}).
\end{theorem}
\begin{proof}
	The theorem follows from the comparison of the KKT conditions of both problems. Assuming that $(\tilde{x},\tilde{u},\tilde{\gamma})$ solves (\ref{eq:mincomp}), we define $\tilde{g}=(\sum\limits_{i\in{\cal D}_1}\tilde{x}_i K_i)\tilde{u} - f^{(1)}$ and $\tilde{\gamma}_1 = (f^{(1)}+\tilde{g})^\top u$, $\tilde{\gamma}_2 = (f^{(2)}-\tilde{g})^\top u$. Then it is straightforward to check that $(\tilde{x},\tilde{u},\tilde{\gamma}_1,\tilde{\gamma}_2,\tilde{g})$ satisfies the KKT conditions of  (\ref{eq:mincomp2}). Now assume that $(\hat{x},\hat{u},\hat{\gamma}_1,\hat{\gamma}_2,\hat{g})$ is a solution of (\ref{eq:mincomp2}). Then $(\hat{x},\hat{u},\hat{\gamma}_1+\hat{\gamma}_2)$ is feasible in (\ref{eq:mincomp}). We know from above that $(\tilde{x},\tilde{u},\tilde{\gamma}_1,\tilde{\gamma}_2,\tilde{g})$ is a solution of (\ref{eq:mincomp2}) with the optimal objective value $\tilde{\gamma}_1+\tilde{\gamma}_2$. Because both problems are equivalent to convex problems (their semidefinite reformulations), then $\tilde{\gamma}=\tilde{\gamma}_1+\tilde{\gamma}_2=\hat{\gamma}_1+\hat{\gamma}_2$ is also the optimal objective value of (\ref{eq:mincomp}), hence $(\hat{x},\hat{u},\hat{\gamma}_1,\hat{\gamma}_2,\hat{g})$ is also optimal for (\ref{eq:mincomp}).\qed
\end{proof}

Using again the Shur complement theorem, we finally arrive at the decomposition of the SDO problem (\ref{eq:SDP}).
\begin{corollary}
Problem (\ref{eq:SDP}) can be equivalently formulated as follows:	
\begin{align}
& \min_{x\in\RR^m,\,\gamma_1\in\RR,\,\gamma_2\in\RR,\,g\in\RR^{n_{\scriptscriptstyle\rm\Gamma}}\ }
\gamma_1 + \gamma_2 \label{eq:ddf}\\
&\quad \mbox{subject to}\nonumber\\
&\quad \sum_{i=1}^m x_i\leq V\nonumber\\
&\quad \underline{x}\leq x\leq\overline{x}\nonumber\\
&\quad  \begin{pmatrix}\gamma_1 & (f^{(1)}+\overleftrightarrow{g})^\top \\
f^{(1)}+\overleftrightarrow{g}\ \  & \displaystyle \sum_{i\in {\cal D}_1}x_i K_i\end{pmatrix}\succeq 0 \nonumber\\
&\quad  \begin{pmatrix}\gamma_2 & (f^{(2)}-\overleftrightarrow{g})^\top \\
f^{(2)}-\overleftrightarrow{g}\ \  & \displaystyle \sum_{i\in {\cal D}_2}x_i K_i\end{pmatrix}\succeq 0\,.\nonumber
\end{align}	
\end{corollary}

Problem (\ref{eq:ddf}) is now exactly the same as problem (\ref{eq:caseb_dec}) arising from arrow decomposition applied to two subdomains.

\section{Numerical experiments}\label{sec:numer}
The decomposition techniques described in the article were applied
to an example whose data (geometry, boundary conditions and forces) are
shown in Figure~\ref{fig:4}--left. 
We always use regular decomposition of the rectangular domain; an example of a decomposition into 8 subdomains is shown in Figure~\ref{fig:4}--right. We have used  finite element meshes with up to 160$\times$80 elements.
\begin{figure}[h]
\begin{center}
 \resizebox{0.4\hsize}{!}
   {\includegraphics{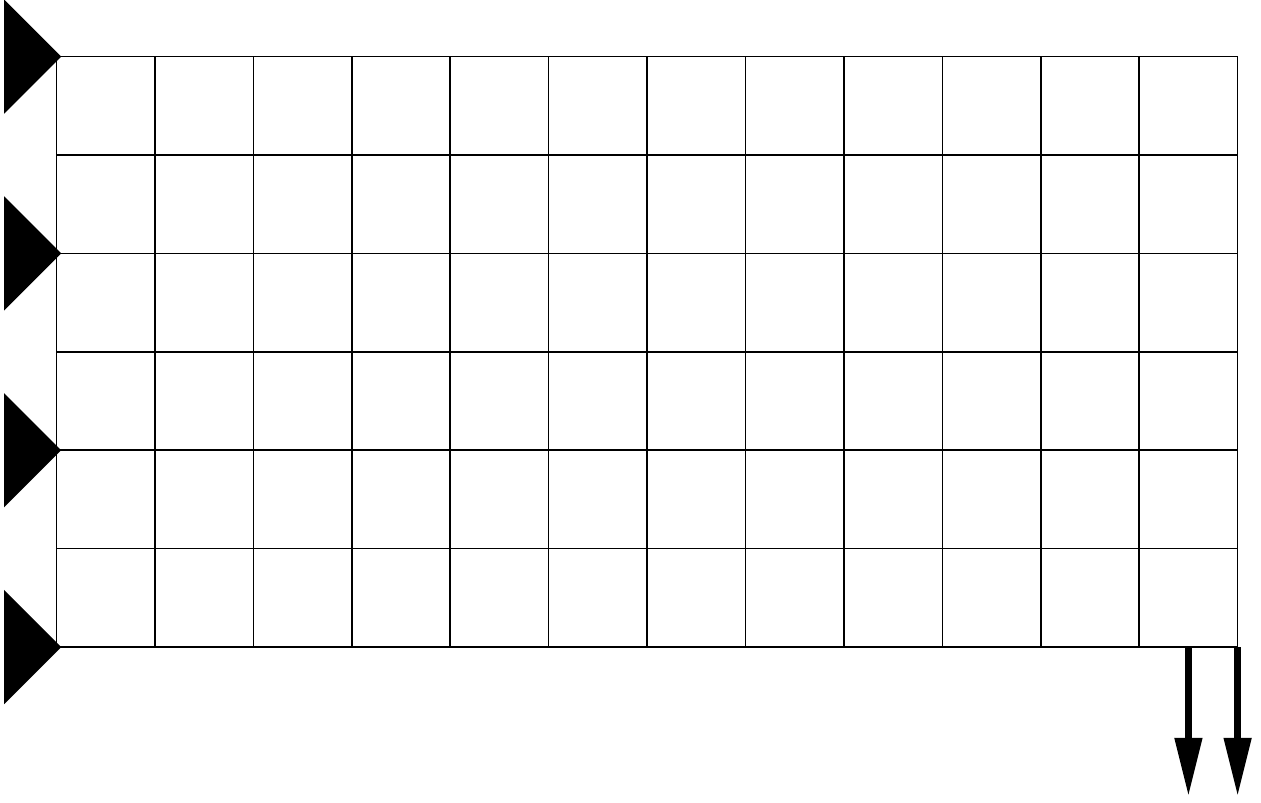}} \ \ 
   \resizebox{0.387\hsize}{!}
   {\includegraphics{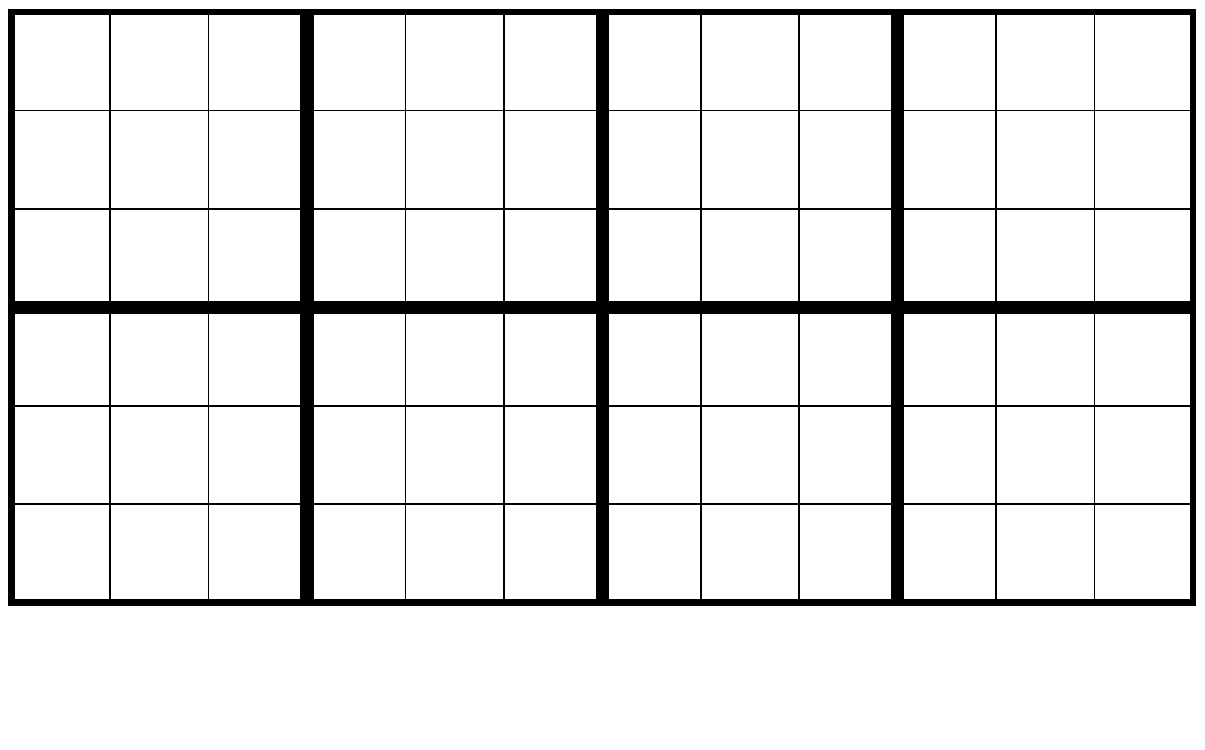}}
  \end{center}
  \caption{Data of numerical examples: geometry, boundary condition and forces (left) and a sample decomposition into 4$\times$2 subdomains (right)}
  \label{fig:4}
  \end{figure}

We tested several codes to solve the SDO problems. Here we present results obtained by MOSEK, version 8.0 \cite{mosek}. The reason for this is that MOSEK best demonstrated the decomposition idea; the speed-up achieved by the decomposition was most significant when using this software. 

When solving the SDO problems, we used default MOSEK settings with the exception of duality gap parameter \verb|MSK_DPAR_INTPNT_CO_TOL_REL_GAP| that was set to $10^{-9}$, instead of the default value $10^{-8}$. We will comment on the resulting accuracy of the solution later in the section.

We also tried to solve the smaller problems by SparseCoLO \cite{sparsecolo}, software that performs the decomposition of matrix constraints based on Theorem~\ref{th:agler} automatically. In particular, the software checks whether the matrix in question has a chordal sparsity graph; if not, the graph is completed to be chordal. After that, maximal cliques are found and Theorem~\ref{th:agler} is applied. Because the sparsity graph of the matrix in problem~(\ref{eq:SDP}) is \emph{not} chordal, a chordal completion is performed by SparseCoLO. Such a completion is not unique and may thus lead to different sets of maximal cliques. And here is the main difference to our approach: while we can steer the decomposition to result in smaller matrix constraints of the same size, matrix constraints resulting from application of SparseCoLO are of variable size, some small, some rather large. This fact has a big effect on the efficiency of SparseCoLO, as we will see in the examples below.

In all experiments we used a 2018 MacBook Pro with 2.3GHz dual-core Intel Core i5, Turbo Boost up to 3.6GHz and 16GB RAM, and MATLAB version 9.2.0 (2017a).

\begin{remark}[Element-wise decomposition] 
	The above text suggests that we always perform decomposition of the original finite element mesh into several (possibly uniform) sub-meshes, each of them having interior points;	
	cf.\ Figures~\ref{fig:mesh}, \ref{fig:ex_mesh}, \ref{fig:4}, and the notation used in Section~\ref{sec:topodec}. However, nothing prevents us from associating each subdomain with a finite element. When every subdomain consist of a single finite element, then the subdomains have no interior points, apart from those lying on the boundary of $\Omega$ and having no neighboring element. For instance, in Example~1, Figure~\ref{fig:ex_mesh}, these would only be degrees of freedom number 31,32,39,40. In the numerical examples below, we will see that the big number of additional variables makes this option less attractive that other decompositions. However, while not the most effective of all decompositions, it is still much less computationally demanding than the original problem. The element-wise decomposition has one big advantage in simplicity of data preparation: the user can use any standard finite-element mesh generator and does not have to worry about definition of subdomains. This may be particularly advantageous in case of highly irregular meshes. 
\end{remark}

\subsection{Computational results}
In the following tables, we present results of the $N_x$$\times$$N_y$ examples using the chordal and arrow decomposition. In these tables the first row of numbers shows data for the original problem (\ref{eq:SDP}), the remaining rows are for the decomposed problems. The first column shows the number of subdomains, the next two ones the number of variables and the size of the largest matrix inequality. After that, we present the total number of iterations needed by MOSEK before it terminated. The next two columns show the total CPU time and CPU time per one iteration and are followed by columns reporting speed-up, both total and per iteration.

In the final column we see the MOSEK constant \verb|MSK_DINF_INTPNT_OPT_STATUS|, a number that is supposed to converge to 1. Let us call this constant $\mu$, for brevity. In our experience, MOSEK delivers acceptable solution reporting ``Solution status: OPTIMAL" when
$$
   0.999 \leq \mu \leq 1.0009\,.
$$
When $\mu$ is farther away from 1, MOSEK, typically in these examples, announces ``Solution status: NEAR\_OPTIMAL." For instance, in the 120$\times$60 example with chordal decomposition with 800 subdomains, MOSEK finished with $\mu=0.9946$ and the final objective value was correct to 3 digits, while with 1800 subdomains MOSEK reported $\mu=0.9865$ and we only got 2 correct digits in the objective function.

We first present results for the 40$\times$20 example using the chordal decomposition; see Table~\ref{tab:1}. 
\begin{table}[hbt]
	\centering
	\caption{Results obtained by MOSEK for the 40$\times$20 example using chordal decomposition.}
	\begin{tabular}{rrr|r|rr|c}
		\addlinespace
		\toprule
		no of & no of & size of &  no of & \multicolumn{ 2}{|c|}{CPU (sec)} &opt\\
doms &   vars &  matrix &  iters & total &         per iter  & status\\ \midrule
   1 &    801 &    1681 &     69 &  1045 &                15 & 0.9999\\ \midrule
   8 &   3523 &     243 &     58 &    31 &              0.53 & 0.9996\\
  32 &   5489 &      73 &     44 &   9.7 &              0.22 & 0.9997 \\
  50 &   6376 &      51 &     46 &   8.8 &              0.19 & 0.9995 \\
 200 &  11243 &      19 &     37 &   6.9 &              0.19 & 0.9987\\
 800 &  24529 &       9 &     35 &    12 &              0.34 & 0.9980\\\bottomrule
	\end{tabular}
\label{tab:1}
\end{table}
%
The table shows that while we increase the number of the subdomains (refine the decomposition), the number of variables increases (those are the additional matrix variables in chordal decomposition) and the size of the constraints decreases. We can further see from Table~\ref{tab:1} that the total number of iterations needed to solve any of the problem formulations is almost constant. The main message of Table~\ref{tab:1} is in the last two columns; here we can see tremendous decrease in the CPU time when solving the decomposed problems.

We now solve the same 40x20 example using the arrow decomposition. The results are presented in Table~\ref{tab:2}. We have added two more columns showing the speed-up, both total and per iteration.
\begin{table}[hbt]
		\centering
		\caption{Results obtained by MOSEK for the 40$\times$20 example using arrow decomposition.}
		\begin{tabular}{rrr|r|rr|rr|c}
			\addlinespace
			\toprule
  no of & no of & size of & no of & \multicolumn{ 2}{|c}{CPU} & \multicolumn{ 2}{|c|}{speed-up} &opt\\
doms &  vars &  matrix & iters & total &          per iter & total &              per iter & status\\ \midrule
   1 &   801 &    1681 &    69 &  1045 &                15 &     1 &                     1 & 0.9999\\ \midrule
   8 &  1032 &     243 &    70 &    28 &              0.40 &    37 &        38 & 0.9999\\
  32 &  1492 &      73 &    63 &   7.6 &              0.12 &   138 &       126 & 1.0003\\
  50 &  1764 &      51 &    64 &   7.1 &              0.11 &   147 &       137 & 0.9999\\
 200 &  3544 &      19 &    51 &   5.1 &              0.10 &   204 &       151 & 0.9999\\ 
 800 &  9204 &       9 &    46 &   6.9 &              0.15 &   150 &       100 & 0.9992\\\bottomrule
		\end{tabular}
	\label{tab:2}
\end{table}
In all examples presented in Table~\ref{tab:2}, MOSEK reported Optimal solution status. Comparing result in Table~\ref{tab:1} and Table~\ref{tab:2}, we can see that the arrow decomposition is not only more efficient than the chordal one, due to smaller number of variables, but also delivers more accurate solution, i.e., a better conditioned SDO problem.

For a comparison, In Table~\ref{tab:3} we present result for example 40$\times$20 obtained by solving problems decomposed by the automatic decomposition software SparseCoLO.
\begin{table}[hbt]
	\centering
	\caption{Results obtained by MOSEK for the 40$\times$20 example using SparseCoLO decomposition.}
	\begin{tabular}{rrr|r|rr|rr}
		\addlinespace
		\toprule
		no of & no of &      size of & no of & \multicolumn{ 2}{|c}{CPU} & \multicolumn{ 2}{|c}{speed-up} \\
		 doms &  vars &       matrix & iters & total &          per iter & total &              per iter \\ \midrule
		   34 & 22997 & 11\dots{}260 &    42 &  301 &                7 &     3 &                     2 \\ \bottomrule
	\end{tabular}
	\label{tab:3}
\end{table}
In this case, the size of the 34 matrix constraints varied from 11 to 260. The decomposed problem is still solved more efficiently that the original one but that speed-up is negligible, compared to either the chordal or the arrow decomposition from Tables~\ref{tab:1} and~\ref{tab:2}.

The next Table~\ref{tab:4a} presents results for the 80$\times$40 discretization and chordal decomposition, while Table~\ref{tab:4} present the results for the same problem using arrow decomposition. This was the largest problem we could solve by MOSEK in the original formulation~\ref{eq:SDP} (due to memory restriction).
\begin{table}[hbt]
	\centering
	\caption{Results obtained by MOSEK for the 80$\times$40 example using chordal decomposition. }
	\begin{tabular}{rrr|r|rr|c}
		\addlinespace
		\toprule
		no of &      no of &    size of &  no of  & \multicolumn{ 2}{|c}{CPU (sec)}\\
		doms &      vars  &     matrix &   iters  &      total &   per iter \\	\midrule
		{ 1} &  { 3201} & { 6561} & { 104} & { 78813} & { 758}  &0.9999\\ \midrule
		{ 8} &{ 12583}& { 883}& { 74}& { 1302}&{ 18} & 0.9992 \\
		32 & 17449 & 243 & 56 & 173 & 3.1 & 0.9993\\
		128 & 24265 &  73 &    51 &     62 &  1.2  &0.9990\\
		{ 200} &{ 27631}& { 51}& 46& 53&{ 1.2} & 0.9993\\
		{ 800} &{ 46873}& { 19}& 40& 41&{ 1.0} & 0.9986 \\
		 3200 &{ 100249}& { 9}& 32& 52&{ 1.6} & 0.9975 \\
		\bottomrule
	\end{tabular}
	\label{tab:4a}
\end{table}
\begin{table}[hbt]
		\centering
		\caption{Results obtained by MOSEK for the 80$\times$40 example using arrow decomposition.}
		\begin{tabular}{rrr|r|rr|rr|c}
			\addlinespace
			\toprule
			no of & no of & size of & no of & \multicolumn{ 2}{|c}{CPU (sec)} & \multicolumn{ 2}{|c|}{speed-up} &opt \\
doms &  vars &  matrix & iters &  total &         per iter & total &              per iter & status\\ \midrule
   1 &  3201 &    6561 &    104 & 78813 &             758 &     1 &                     1 &0.9999\\ \midrule
   8 &  3632 &     883 &    88 &   1098 &             12.5 &    72 &                    61 & 0.9999\\
  32 &  4412 &     243 &    83 &    121 &              1.5 &  651 &                   520 & 0.9999 \\
 128 &  6308 &      73 &    69 &     25 &              0.4 &  3153 &                  2092 &0.9999\\
 200 &  7424 &      51 &    65 &     18 &              0.3 &  4379 &                  2737 &0.9999\\
 800 & 14864 &      19 &    62 &     17 &              0.3 & 4636 &                  2764 &0.9999\\
3200 & 37604 &       9 &    44 &     25 &              0.6 &  3153 &                 1334 &0.9999\\ \bottomrule
		\end{tabular}
	\label{tab:4}
\end{table}
As we can see, for a larger problem the speed-up obtained by arrow decomposition is even more significant.

Examples with finer discretization cannot be solved by MOSEK in the original formulation~\ref{eq:SDP} (on the laptop we used for the experiments). They can, however, easily be solved in the decomposed setting. The results are presented in the next tables. In these tables, we also show estimated number of iterations and CPU time for the original problem; these numbers are extrapolated from the lower-dimensional problems (also those that are not presented here).

Table~\ref{tab:5} presents results for the 120$\times$60 discretization and chordal decomposition, while Table~\ref{tab:6} shows the results for the same example, this time using arrow decomposition.
\begin{table}[hbt]
	\centering
	\caption{Results obtained by MOSEK for the 120$\times$60 example using chordal decomposition. Iteration count and CPU time in the first row are estimated and marked by the $\dag$ symbol.}
	\begin{tabular}{rrr|r|rr|c}
		\addlinespace
		\toprule
		no of &      no of &    size of &  no of  & \multicolumn{ 2}{|c|}{CPU (sec)} & opt\\
		doms &      vars  &     matrix &   iters  &      total &   per iter & status \\	\midrule
		{ 1} &  { 7200} & { 14641} & { 139$^\dag$} & { 1045932$^\dag$} & { 7524} & 0.9999  \\ \midrule
		{ 200} &{ 51539}& { 19}& { 60}& { 236}&{ 3.9} & 0.9950 \\
		{ 800} &{ 76977}& { 19}& { 50}& { 129}&{ 2.6} & 0.9946\\
		{ 1800} &{ 106903}& { 19}& 47& { 114}&{ 2.4} & 0.9865 \\
		\bottomrule
	\end{tabular}
\label{tab:5}
\end{table}
\begin{table}[hbt]
	\centering
	\caption{Results obtained by MOSEK for the 120$\times$60 example using arrow decomposition. Iteration count and CPU time in the first row are estimated and marked by the $\dag$ symbol.}
	\begin{tabular}{rrr|r|rr|rr|c}
		\addlinespace
		\toprule
		no of &      no of &    size of &  no of  & \multicolumn{ 2}{|c}{CPU (sec)}& \multicolumn{ 2}{|c|}{speed-up}& opt\\
   doms &      vars  &     matrix &   iters  &      total &   per iter& total &   per iter & status\\	\midrule
 	  1   &    { 7200} & { 14641} &   { $^\dag$139} & { $^\dag$1045932} & { 7525} & { 1} & { 1}    & 0.9999\\ \midrule
	 50  &       9524 &        339 &        96 &      524 &      5.5 &   1996 & 1379 & 0.9996\\ 
	200 &      12904 &         99 &        82 &        89 &       1.1 &  11752 & 6933 & 0.9997\\
	450 &      16984 &         51 &        82 &        55 &       0.67 &  19017 & 11219 &0.9997  \\ 
	800 &{     21764}& {       33}& {      71}& {       37}&{      0.52}&{ 28268}&{ 14439}&0.9997\\ 
   1800 &      33424 &         19 &        65 &         42 &       0.65 &  24903 & 11645 & 0.9998 \\
   7200 &      85204 &          9 &        55 &        90 &       1.6 &  11621 &  4598 & 0.9997  \\
		\bottomrule
	\end{tabular}
\label{tab:6}
\end{table}
When using the chordal decomposition (Table~\ref{tab:5}), MOSEK had significant problems with convergence to the optimal solution. In case of 800 subdomains, the final objective value was correct to 3 digits, while for the 1800 subdomains only to 2 digits. In both cases, the solution status of MOSEK was ``Nearly optimal". In case of arrow decomposition, all problems finished with "Optimal" solution status. Again, the arrow decomposition outperforms the chordal one, so from now on we will only focus on the arrow decomposition.

From the results presented so far, it seems that the most efficient decomposition is either the finest or the second-finest one (not counting the element-wise decomposition); in the first case, each subdomain contains four finite elements, in the second case 16 finite elements. To get a clearer idea about the relation of the problem size and speed-up, we present the next Table~\ref{tab:7} of results for examples with  dimension increasing from 40$\times$20 to 160$\times$80 elements. For each example we only consider the finest decomposition with four finite elements per subdomain. So the size of every matrix inequality is always at most 19. The CPU times for original formulation of the larger problems have been extrapolated and are denoted by the $^\dag$ symbol.
\begin{table}[hbt]
	\centering
	\caption{Results obtained by MOSEK using arrow decomposition. Symbol $^\dag$ denotes extrapolated CPU times.}
	\begin{tabular}{r|rrr|rrrc|r}
		\addlinespace
		\toprule
		&    \multicolumn{ 3}{|c}{ORIGINAL}&   \multicolumn{ 4}{|c|}{DECOMPOSED} & speed-up \\
		problem   & no of & size of & CPU   &  no of & {\footnotesize size of}     & CPU    &opt   \\
		& vars  & matrix  & total &  vars  &  {\footnotesize matrix}  & total   & status     \\
		\midrule
		40$\times$20  &  801 &	 1681 &	              1045 &  3544 & 19 & 	5 & 0.9999 & { 204} \\
		60$\times$30  & 1801 &	 3721 &              12468 &  8164 & 19 &  9 & 0.9999  & { 1370}\\
		80$\times$40  & 3201 &	 6561 &	            78813 & 14684 & 19 &  17 & 0.9999  & { 4636}\\
		100$\times$50 & 5001 &	10201 &	  $^\dag$312560   & 23104 & 19 &  25 & 0.9999  & { 12502}\\
		120$\times$60 & 7201 &	14641 &	  $^\dag$1045932   & 33424 & 19 &  42 & 0.9998  & { 24903}\\
		140$\times$70 & 9801 & 19881 &  $^\dag$2900382   & 45664 & 19 & 59 & 0.9994  & { 49159}\\
		160$\times$80 &12801 & 25921 &  $^\dag$7003213   & 59764 & 19 & 74 & 0.9984  & { 94638}\\  \midrule
		\multicolumn{ 3}{l}{{ complexity}
			$c\cdot$size$^q$}&\multicolumn{ 1}{r}{$q=3.18$}&&\multicolumn{ 2}{r}{$q=1.0006$}&\multicolumn{ 1}{r}{}\\
		\bottomrule
	\end{tabular}
\label{tab:7}
\end{table}

The last row of Table~\ref{tab:7} presents the estimate of computational complexity of each approach, as a function $c \nu^q$ of problem size $\nu$; in this case, $\nu$ is the number of variables of the SDO problem, as reported in the table. The exponent $q$ is estimated from the CPU times. In case of the original, undecomposed problem, we calculated $q\approx 3.18$ which slightly underestimates the theoretical complexity of interior point methods for SDO. The decomposed problem, on the other hand, exhibits linear complexity with $q\approx 1.0006$. See also Figure~\ref{fig:33} for graphical representation of the complexity of the original problem (top line), single iteration of the original problem (middle line) and of the decomposed problem (bottom line). \emph{This, in our opinion, is the principal contribution of the arrow decomposition method.}
\begin{figure}[h]
	\begin{center}
		\resizebox{0.7\hsize}{!}
		{\includegraphics{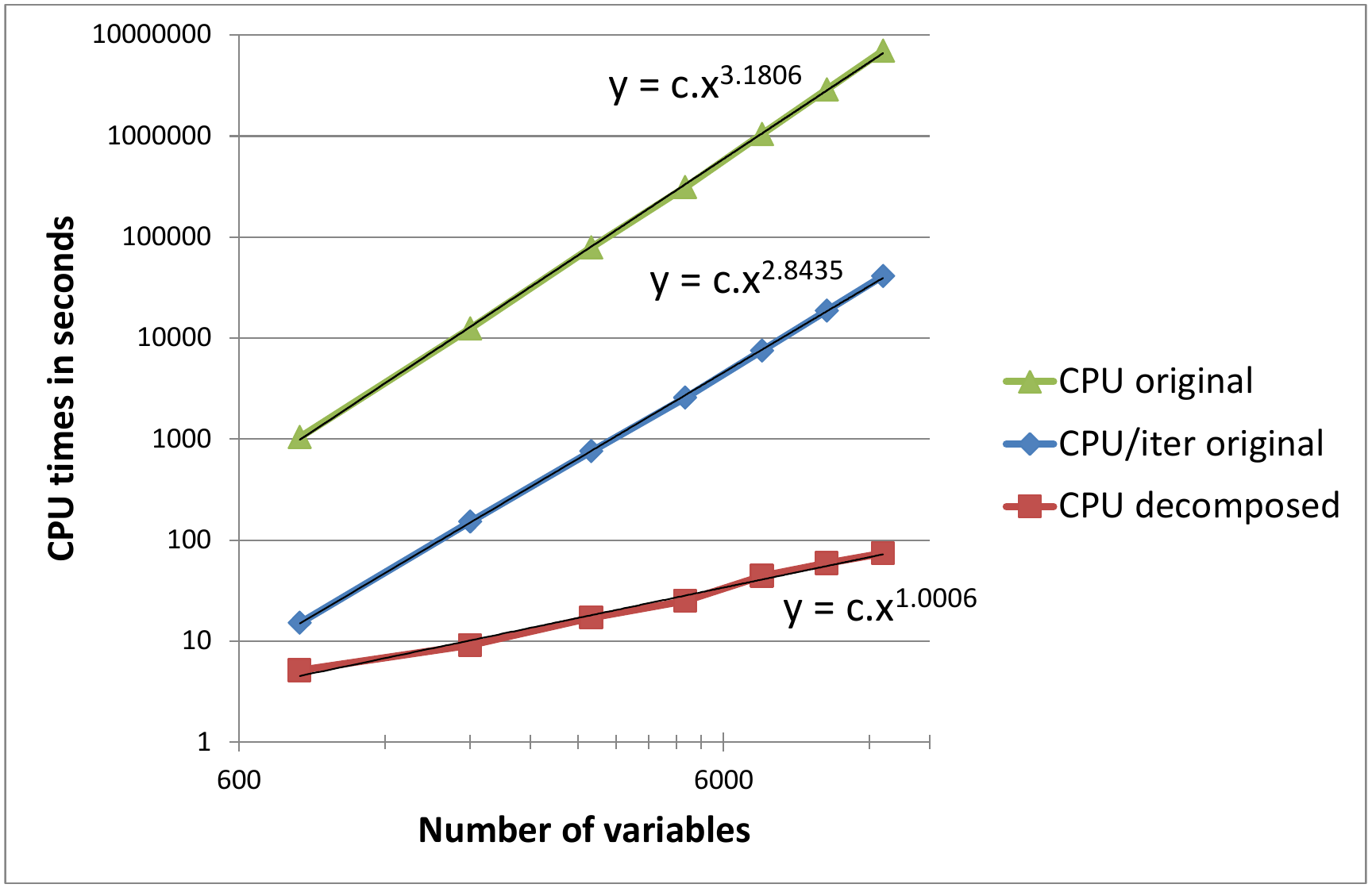}}
	\end{center}
	\caption{Complexity of the original problem (top), of a single iteration in the original problem (middle) and of the decomposed problem (bottom).}
	\label{fig:33}
\end{figure}

\section*{Acknowledgment}
The author would like to thank Masakazu Kojima for discussions on chordal decomposition of the topology optimization problem. The work on this article was initiated while the author was
visiting the Institute for Pure and Applied Mathematics, UCLA. The
support and friendly atmosphere of the Institute are acknowledged with
gratitude.

\bibliographystyle{spmpsci}
\bibliography{sdp_dd}

\end{document}